\documentclass[a4paper,12pt,leqno]{article}

\usepackage{amssymb,amsmath}
\usepackage[abbrev]{amsrefs}
\usepackage{xy}
\xyoption{all}

\newtheorem{defn}{Definition}[section]

\newtheorem{corollary}[defn]{Corollary}
\newtheorem{rem}[defn]{Remark}
\newtheorem{exm}[defn]{Example}
\newtheorem{lemma}[defn]{Lemma}
\newtheorem{theorem}[defn]{Theorem}
\newtheorem{notat}[defn]{Notation}
\newtheorem{newpar}[defn]{}

\newtheorem{xdefn}{Definition.}
\newtheorem{xproposition}{Proposition.}
\newtheorem{xcorollary}{Corollary.}
\newtheorem{xrem}{Remark.}
\newtheorem{xexm}{Example.}
\newtheorem{xlemma}{Lemma.}
\newtheorem{xtheorem}{Theorem.}
\newtheorem{xnotat}{Notation.}
\newtheorem{xnewpar}{\it}
\newtheorem{xproof}{{\it Proof. }}
\newtheorem{xproofof}{{\it Proof}}

\newenvironment{definition}{\begin{defn}\em}{\end{defn}}
\newenvironment{remark}{\begin{rem}\em}{\end{rem}}
\newenvironment{example}{\begin{exm}\em}{\end{exm}}

\newenvironment{proof}{\begin{xproof}\em}{\end{xproof}}

\newenvironment{newparagraph*}[1]{\begin{xnewpar}\hspace*{-1.5mm}{#1}. \rm}{\end{xnewpar}}

\newenvironment{definition*}{\begin{xdefn}\em}{\end{xdefn}}
\newenvironment{remark*}{\begin{xrem}\em}{\end{xrem}}
\newenvironment{example*}{\begin{xexm}\em}{\end{xexm}}
\newenvironment{notation*}{\begin{xnotat}\em}{\end{xnotat}}
\newenvironment{proposition*}{\begin{xproposition}}{\end{xproposition}}
\newenvironment{corollary*}{\begin{xcorollary}}{\end{xcorollary}}
\newenvironment{lemma*}{\begin{xlemma}}{\end{xlemma}}
\newenvironment{theorem*}{\begin{xtheorem}}{\end{xtheorem}}

\def\qed{\hspace{0.3cm}{\rule{1ex}{2ex}}}
\newcommand\V{\bigvee}
\newcommand\ie{i.e.}
\newcommand\eg{e.g.}
\newcommand\st{\mid}
\newcommand\cf{\textrm{cf.}}
\newcommand\downsegment{{\downarrow}}

\newcommand\spp{\varsigma}

\newcommand\Max{\operatorname{Max}}
\newcommand\pwset[1]{\wp(\,#1)}
\newcommand\opens{\operatorname{\mathcal{O}}}

\newcommand\topology{\operatorname{\Omega}}

\newcommand\Mod{\mathbf{Mod}}
\newcommand\Frm{\mathbf{Frm}}
\newcommand\Loc{\mathbf{Loc}}
\newcommand\GTop{\mathbf{GTop}}
\newcommand\Etale{\mathbf{LH}}

\newcommand\opp[1]{{#1}^{\textrm{op}}}

\newcommand\ident{\mathrm{id}}
\newcommand\ipi{\mathcal I}

\newcommand\CC{\mathbb{C}}
\newcommand\SL{\mathbf{SL}}

\newcommand\inner[2]{{\langle #1,#2 \rangle}}

\newcommand\hide[1]{}
\newcommand\qseq[2]{{\lbrack\!\lbrack} {#1}={#2}{\rbrack\!\rbrack}}
\newcommand\vect[1]{\boldsymbol{ #1}}
\newcommand\sets{\operatorname{\mathbf{Set}}}
\newcommand\sections{\mathit{\Gamma}}

\newcommand\hmb{\mathbf{HMB}}
\newcommand\sh{\mathbf{Sh}}

\newcommand\pb{\operatorname{{\times}_0}}
\newcommand\tensover{\operatorname{{\otimes}_0}}

\newcommand\Rel{\operatorname{\mathbf{Rel}}}
\newcommand\Map{\operatorname{\mathbf{Map}}}

\newcommand{\relto}{{\longrightarrow\hspace*{-2.8ex}{\mapstochar}\hspace*{2.6ex}}}

\newcommand\maps{\operatorname{Map}}

\newenvironment{eq}{\setcounter{equation}{\arabic{defn}}\begin{equation}}{\end{equation}\setcounter{defn}{\arabic{equation}}}
\newenvironment{eqarray}{\setcounter{equation}{\arabic{defn}}\begin{eqnarray}}{\end{eqnarray}\setcounter{defn}{\arabic{equation}}}

\begin{document}

\title{Groupoid sheaves as quantale sheaves\thanks{Research supported in part by Funda\c{c}\~{a}o para a Ci\^{e}ncia e a Tecnologia through FEDER.}
}

\author{\sc Pedro Resende}

\date{~}

\maketitle

\begin{abstract}
Several notions of sheaf on various types of quantale have been proposed and studied in the last twenty five years. It is fairly standard that for an involutive quantale $Q$ satisfying mild algebraic properties the sheaves on $Q$ can be defined to be the idempotent self-adjoint $Q$-valued matrices. These can be thought of as $Q$-valued equivalence relations, and, accordingly, the morphisms of sheaves are the $Q$-valued functional relations. Few concrete examples of such sheaves are known, however, and in this paper we provide a new one by showing that the category of equivariant sheaves on a localic \'etale groupoid $G$ (the classifying topos of $G$) is equivalent to the category of sheaves on its involutive quantale $\opens(G)$. As a means towards this end we begin by replacing the category of matrix sheaves on $Q$ by an equivalent category of complete Hilbert $Q$-modules, and we approach the envisaged example where $Q$ is an inverse quantal frame $\opens(G)$ by placing it in the wider context of stably supported quantales, on one hand, and in the wider context of a module theoretic description of arbitrary actions of \'etale groupoids, both of which may be interesting in their own right.\\
\vspace*{-2mm}~
\vspace*{-2mm}~
\end{abstract}

\newpage
{\small{\tableofcontents}}

\section{Introduction}\label{introduction}

By an involutive quantaloid is meant a sup-lattice enriched category (or sometimes a non-unital generalization of this) equipped with an order two sup-lattice enriched automorphism $(-)^*:Q^{\mathrm{op}}\to Q$. Involutive quantales are the involutive quantaloids with a single object. These structures can be regarded as generalized topological spaces in their own right, and several papers devoted to notions of sheaf on them have been published. In most of these the sheaves are, one way or the other, defined to be quantale-valued or quantaloid-valued matrices in a way that generalizes the notion of frame-valued set of \citelist{ \cite{Higgs} \cite{FS} } (see also \citelist{\cite{Borceux}*{Section 2} \cite{elephant}*{pp.\ 502--513}}). The appropriateness of such definitions of sheaf is emphasized by the fact that every Grothendieck topos is, up to equivalence, the category of sheaves on an involutive quantaloid of ``binary relations'' that is obtained directly from the topos or from a site \cite{Allegories,W82,CW,Pitts}. In essence, this fact is the source of all the concrete examples known so far of toposes of (set valued) sheaves on involutive quantaloids.

In this paper we give another example, based on the correspondence between \'etale groupoids and quantales \cite{Re07}, showing that the topos of equivariant sheaves on a localic \'etale groupoid $G$ is equivalent to the category of sheaves on the involutive quantale $\opens(G)$ of the groupoid. This provides a way, via the groupoid representation of toposes \citelist{ \cite{JT} \cite{elephant}*{\S C5}}, in which \'etendues arise as categories of sheaves on involutive quantales.

We shall begin with a somewhat detailed survey on quantaloid sheaves, meant both to provide the reader with a reasonable view of the literature and also to help trace the origins of the notion of sheaf that we shall consider in this paper.

\paragraph{Sheaves and Grothendieck toposes.}

Locales are a point-free substitute for topological spaces \cite{pointless,stonespaces}, in particular providing a generalization of sober topological spaces, such as Hausdorff spaces. However, in many ways this is only a mild generalization. A much broader one is given by sites, where suitable families of morphisms generalize the role of open covers. A Grothendieck topos is by definition the category (up to equivalence) of sheaves on a site \cite{maclanemoerdijk,elephant,Borceux}. Whereas locales are determined up to isomorphism by their toposes of sheaves, it is the case that very different sites may yield equivalent toposes, and it is the toposes, rather than the sites, that provide the right notion of generalized space. In particular, the category of locales $\Loc$ is equivalent to a reflective full subcategory of the category of Grothendieck toposes $\GTop$, and the
role played by toposes as spaces in their own right is emphasized by the existence of invariants such as homotopy and cohomology for them \cite{topostheory}. Indeed, the fact that toposes provide good grounds on which to study cohomology is precisely one the reasons that led Grothendieck to considering them \cite{SGA4}.

A ``measure'' of how far toposes really differ from locales is provided by the representation theorem of Joyal and Tierney \citelist{ \cite{JT} \cite{elephant}*{\S C5}}: a Grothendieck topos is, up to equivalence, the category of equivariant sheaves on an open groupoid $G$ in $\Loc$, in other words the sheaves on the locale of objects $G_0$ that in addition are equipped with an action of $G$. We shall refer to these simply as $G$-sheaves. In particular, every \'etendue can be represented by an \'etale groupoid \cite{JT}*{Th.\ VIII.3.3} (see also \citelist{ \cite{SGA4}*{\S IV.9.8.2} \cite{KoMoer}}). Similarly to the situation with sites, groupoids themselves should be regarded as presentations of spaces rather than spaces themselves, since two groupoids yield equivalent toposes if and only if they are Morita equivalent \cite{Moer88,Moer90}. For suitable open groupoids (the \'etale-complete ones, in particular the \'etale groupoids), the topos $BG$ of equivariant sheaves on a groupoid $G$ is called the classifying topos of $G$ because it classifies principal $G$-bundles \cite{Moer88,Bunge}. It can also be thought of as the space of orbits of $G$, since it arises as a colimit, in the appropriate bicategorical sense, of the nerve of $G$ regarded as a simplicial topos \cite{Moer88}.

A different way to represent Grothendieck toposes (and also more general categories) stems from work on ``categories of relations'' by various authors. For instance, Freyd's work on allegories \citelist{ \cite{Allegories} \cite{elephant}*{\S A3}} dates back to the seventies and already contains many ingredients of later theories of sheaves on involutive quantaloids. One is the equivalence of categories $\mathcal E\simeq\Map(\Rel(\mathcal E))$, where $\Rel(\mathcal E)$ is the involutive quantaloid of binary relations on a Grothendieck topos $\mathcal E$ and $\Map(\Rel(\mathcal E))$ its category of ``maps'', which are the arrows $f$ in $\Rel(\mathcal E)$ satisfying $ff^*\le 1$ and $f^*f\ge 1$. Another idea is the completion \cite[\S 2.226]{Allegories} that replaces a quantaloid by a suitable category of matrices that play the role of sheaves.

A different characterization of categories of relations is that of Carboni and Walters \cite{CW}, based on which Pitts \cite{Pitts} has obtained an adjunction between the 2-category of Grothendieck toposes and the 2-category of quantaloids known as distributive categories of relations (dcr's for short). The construction of the category of sheaves on a dcr by matrices is analogous to  completion for allegories in \cite[\S 2.226]{Allegories}. This also parallels the representation of toposes by localic groupoids: dcr's are another generalization of locales because a dcr is a sup-lattice enriched category $A$ equipped with a suitable product $A\otimes A\to A$ that turns it into a ``locale with many objects'' (a fancy name would be ``localoid'') in the sense that if $A$ has only one object then it is a locale and the product is $\wedge$ (we also note that involution and intersection are not primitive operations as in an allegory); every Grothendieck topos arises as the category of sheaves on a dcr.

Still another representation of Grothendieck toposes by involutive quantaloids can be found in the work of Walters \cite{W82}, which shows how to associate an involutive quantaloid $B$ to any site in such a way that the topos is equivalent to the category of sheaves on $B$. Here the notion of ``sheaf'' is that of skeletal Cauchy-complete $B$-enriched category. Similarly to the previous examples, such sheaves can again be regarded as matrices and again they generalize frame-valued sets, but there are some technical differences, in particular the fact that they must lie above identity matrices. This difference can be quickly described in the case of sheaves on a locale $A$: whereas $A$-valued sets directly model sheaves on $A$, Walters' approach uses categories enriched in the quantaloid $B$ that arises as the split idempotent completion of $A$ \cite{W81}.

\paragraph{Quantales and noncommutative topology.}

Quantales surface in a wide range of mathematical subjects, including algebra, analysis, geometry, topology, logic, physics, computer science, etc. Technically they are just sup-lattice ordered semigroups, but the name ``quantale'' (coined by Mulvey in the eighties) is associated with the idea of ``quantizing'' the point free spaces of locale theory, in analogy with the generalization of Hausdorff topology via noncommutative C*-algebras \cite{Giles,GK,Akem70,Akem71,Connes}. In particular, Mulvey's original idea was precisely to define the ``non-commutative spectrum'' of a C*-algebra to be a suitable quantale \cite{M86}. Several variants of this idea have been pursued for C*-algebras \cite{BRB,Rosicky,MP1,MP2,KPRR,KR} and arbitrary rings \cite{Rose92,Rosenthal1,BB90,BC94}, often involving classes of quantales that satisfy specific properties, along with purely algebraic investigations of the spatial aspects of quantales \cite{K02,PR}. More recent examples include the space of Penrose tilings \cite{MR} and groupoids \cite{Re07}.

The idea that quantales can be regarded as spaces leads naturally to the question of how such spaces can be studied, for instance via which cohomology theories, etc. While there has never been a systematic pursuit in this direction, there has nevertheless been some effort aimed at finding good notions of sheaf on quantales or, more generally, on quantaloids.
This effort is justified by a number of reasons. One, of course, is the possibility of using sheaves in order to define topological invariants for quantales and for the objects they are associated with. Another motivation for looking at quantale or quantaloid sheaves is the role they may play as structure sheaves when studying spectra of noncommutative rings or C*-algebras \cite{BC94,BB90,VdB}; in particular, there is interest in understanding more about the sheaves on the quantale $\Max A$ of a C*-algebra $A$ due to the hope that a suitable notion of structure sheaf may do away with the excess of quantale homomorphisms that make $\Max$ a somewhat ill-behaved functor \cite{KR}. Also interesting is the possibility of obtaining useful extensions of the {\it K}-theory sheaves of \cite{Dadarlat} in the context of the program of classification of C*-algebras.

Despite a few exceptions \cite{BBSV,BB86,BC93}, most of the papers on sheaves for quantales or quantaloids are to a greater or lesser extent based on the definition of sheaves as matrices (frame-valued sets) introduced originally for sheaves on locales. Many technicalities depend, of course, on the specific types of quantales or quantaloids under consideration. For instance, there have been direct generalizations of frame-valued sets for right-sided or left-sided quantales \cite{NawazPhD,MN,BBS,G94}, which are those that model quantales of right-sided (resp.\ left-sided) ideals of rings. There have also been proposals for definitions of sheaf on general quantaloids \cite{G00,VdB,Rosenthal2,G95} that carry a generalization of frame-valued sets, but it has been shown by Borceux and Cruciani \cite{BC98} that such definitions applied to general non-involutive quantales should at most give us a notion of ordered sheaf, rather than a discrete one. Subsequent work by Stubbe \cite{IsarPhD,S05} adheres to this point of view while also generalizing the bicategorical enrichment approach of Walters \cite{W81}.

For (discrete) sheaves one needs the base quantale or quantaloid $Q$ to be equipped with an involution in order to be able to define self-adjoint $Q$-valued matrices (equivalently, in order to define a ``symmetric'' notion of $Q$-valued equality). For instance, the quantaloids that appear in the early works on categories of relations mentioned above are involutive, and the matrices that correspond to sheaves on them are self-adjoint. For more general involutive quantales and quantaloids a fairly stable theory of sheaves seems to be emerging, with many basic definitions being close in spirit, if not in form. For instance, sheaves on involutive quantaloids have been studied by Gylys \cite{G01}. Slightly later Garraway \cite{Garraway,GarrawayPhD} extended the theory of sheaves on the dcr's of \cite{Pitts} to a rather general class of non-unital quantaloids, while Mulvey and Ramos \cite{MulJoel,JoelPhD} have produced a theory for involutive quantales directly inspired by the axiomatic approach of \cite{FS,MN}. More recently, Heymans \cite{Heymans} has provided a study of sheaves based on quantaloid enriched categories in the style of Walters \cite{W81,W82}, making the connections to \cite{G01} explicit and leading to a representation theorem for Grothendieck toposes by so-called Grothendieck quantales \cite{HeymansPhD}.

While the theory of sheaves on involutive quantaloids appears to be thriving, a negative aspect should nevertheless be mentioned, namely that so far the increase in generality of the theory has not been accompanied by a corresponding rise in the number of known examples.

\paragraph{Overview of the paper.}

Each \'etale groupoid $G$ has an associated involutive quantale $\opens(G)$ \cite{Re07} (for a topological groupoid this quantale is just the topology equipped with pointwise multiplication of groupoid arrows), and the natural question of how the $G$-sheaves relate to notions of sheaf on $\opens(G)$ arises. The main aim of this paper is to provide an answer to this question. As we shall see, the conclusion is that the ``standard'' category of matrix sheaves on an involutive quantale such as $\opens(G)$ (by which we shall mean the category of $\opens(G)$-sets as, say, in \cite{Garraway}) and the classifying topos $BG$ are isomorphic. Hence, \'etale groupoid sheaves yield a new example of matrix sheaves on involutive quantales. 

We shall proceed in three steps:

\begin{description}
\item[Step 1:] We show, in section \ref{sec:gaqm}, how the actions of an open or \'etale groupoid $G$ can be described in terms of modules on the quantale $\opens(G)$. The main result of this section is theorem \ref{lem:fullfaithfullh}, which proves that the category $G$-$\Loc$ of actions of an \'etale groupoid $G$ is equivalent to a suitable category of modules on $\opens(G)$, whose algebraic description is quite simple. As a restriction of this we obtain a definition of sheaf on $\opens(G)$ in terms of quantale modules, and two categories of quantale modules, $\opens(G)$-$\Etale$ and $\opens(G)$-$\sh$, which are isomorphic to $BG$.
\item[Step 2:] We shall recall the basics of the theory of quantale-valued sets and show that these can be replaced by the theory of Hilbert modules \cite{Paseka,Paseka2} equipped with Hilbert bases in a way that generalizes the work of \cite{RR} for sheaves on locales. This is the contents of section \ref{sec:mvhb}, whose main result is theorem \ref{equivrelQQhmb}, from which it follows, for an arbitrary involutive quantale $Q$, that the category $\sets(Q)$ of $Q$-sets is equivalent to the category of maps of the involutive quantaloid $Q$-$\hmb$ of Hilbert $Q$-modules with Hilbert bases. 
\item[Step 3:] Finally, in section \ref{sec:backgrpdshvs} we show that the objects of the categories $\opens(G)$-$\Etale$ and $\opens(G)$-$\sh$ coincide with the Hilbert $\opens(G)$-modules with Hilbert bases (theorems \ref{thm:Qsheaf} and \ref{thm:sheaveseqetales}). Hence, in particular, $\opens(G)$-$\sh$ coincides with the category of maps of $\opens(G)$-$\hmb$ (lemma \ref{lem:main}), and therefore $BG$ is equivalent to $\opens(G)$-$\sets$ (theorem \ref{thm:main}).
\end{description}

We hope these results provide further evidence of what should be considered a ``good'' notion of sheaf for involutive quantales and quantaloids, and we provide a brief discussion of this at the end of section \ref{sec:end}.

Besides the main results, the paper contains subsidiary aspects of independent interest: we provide a comparison, in section \ref{sec:supp}, between supported quantales and modular quantales; in section \ref{sec:actions}, a corollary of our results is a proof of the multiplicativity of inverse quantal frames that is simpler than the original one in \cite{Re07} --- in particular not using the representation of inverse quantal frames by inverse semigroups; and, in order to convey a sense of the robustness of the notion of quantale sheaf that we assume in this paper we include, as an appendix (section \ref{sec:shmat}), a brief survey of some of the variants of quantale-valued set that can be found in the literature along with the relation between complete quantale-valued sets and Hilbert modules.

\section{Background}

\subsection{Preliminaries}

The purpose of this section is mostly to recall some definitions and examples concerning locales, localic groupoids and involutive quantales and quantaloids, and to set up notation and terminology.

\paragraph{Locales.}

We shall adopt the same conventions regarding notation and terminology for locales that are used in \cite{Re07}. In particular, following \cite{stonespaces}, we shall denote the category of \emph{frames} and \emph{frame homomorphisms} by $\Frm$. We shall adopt the terminology \emph{locale} instead of frame when referring to objects of the dual category $\opp\Frm$, which we denote by $\Loc$ and whose arrows we refer to as \emph{continuous maps}, or simply \emph{maps}, of locales. If $X$ is a locale we shall usually write $\opens(X)$ for the same locale regarded as an object of $\Frm$.
If $f:X\to Y$ is a map of locales we shall refer to the frame homomorphism $f^*:\opens(Y)\to\opens(X)$ that defines it as its \emph{inverse image}. If $f$ is an \emph{open map}, the left adjoint to $f^*$ is referred to as the \emph{direct image} of $f$ and it is denoted by $f_!:\opens(X)\to\opens(Y)$. The product of $X$ and $Y$ in $\Loc$ is denoted by $X\times Y$. It coincides with the coproduct of $\opens(X)$ and $\opens(Y)$ in $\Frm$, which is the tensor product in the category of sup-lattices $\SL$ \cite[\S I.5]{JT}. Hence, we write $\opens(X\times Y)=\opens(X)\otimes\opens(Y)$.

\paragraph{Groupoids.} A \emph{groupoid} in a category $C$ with enough pullbacks is an internal groupoid in $C$. We denote the locales of objects and arrows of a groupoid $G$ respectively by $G_0$ and $G_1$, and we adopt the following notation for the structure maps
\[G\ \ \ \ =\ \ \ \ \xymatrix{
G_2\ar[rr]^-m&&G_1\ar@(ru,lu)[]_i\ar@<1.2ex>[rr]^r\ar@<-1.2ex>[rr]_d&&G_0\ar[ll]|u
}\;,\]
where $G_2$ is the pullback of the \emph{domain} and \emph{range} maps:
\[\vcenter{\xymatrix{G_2\ar[r]^{\pi_1}\ar[d]_{\pi_2}&G_1\ar[d]^r\\
G_1\ar[r]_d&G_0}}\;.\]
We remark that, since $G$ is a groupoid rather than just an internal category, the \emph{multiplication map} $m$ is the pullback of $d$ along itself:
\[\vcenter{\xymatrix{G_2\ar[r]^{\pi_1}\ar[d]_{m}&G_1\ar[d]^d\\
G_1\ar[r]_d&G_0}}\;.\]
The following are examples:
\begin{itemize}
\item A \emph{topological groupoid} is an internal groupoid in the category of topological spaces and continuous maps.
\item A \emph{Lie groupoid} is an internal groupoid in the category of smooth manifolds such that $d$ is a submersion (this condition ensures that the pullback $G_2$ exists).
\item The category $\Loc$ has pullbacks and a \emph{localic groupoid} is an internal groupoid in $\Loc$.
\end{itemize}
A localic groupoid $G$ is said to be \emph{open} if $d$ is an open map. Hence, if $G$ is open the multiplication map $m$ is also an open map. An \emph{\'etale groupoid} is an open groupoid such that $d$ is a local homeomorphism, in which case all the structure maps are local homeomorphisms and, hence, $G_0$ is isomorphic to an open sublocale of $G_1$. Conversely, any open groupoid for which $u$ is an open map is necessarily \'etale \cite[Corollary 5.12]{Re07}. Similar conventions and remarks apply to topological groupoids.

\paragraph{Involutive quantales.} By an \emph{involutive quantale} is meant an involutive semigroup in the monoidal category $\SL$ of sup-lattices. We shall adopt the following terminology and notation:

\begin{itemize}
\item The product of two elements $a$ and $b$ of an involutive quantale $Q$ is denoted by $ab$, the involute of $a$ is denoted by $a^*$, the join of a subset $S\subset Q$ by $\V S$, the top element by $1_Q$ or simply $1$, and the bottom element by $0_Q$ or simply $0$. The elements such that $a^*=a$ are \emph{self-adjoint}. The idempotent self-adjoint elements are the \emph{projections}.

\item The involutive quantale $Q$ is \emph{unital} if there is a unit for the multiplication, which is denoted by $e_Q$ or simply $e$. (This is necessarily a projection.)

\item By a \emph{homomorphism} of involutive quantales $h:Q\to R$ is meant a homomorphism of involutive semigroups in $\SL$. If $Q$ and $R$ are unital, the homomorphism $h$ is \emph{unital} if $h(e_Q)=e_R$.
\end{itemize}

Similarly, given an involutive quantale $Q$, by a \emph{(left) $Q$-module} will be meant a sup-lattice $M$ equipped with an associative left action $Q\otimes M\to M$ in $\SL$ (the involution of $Q$ plays no role). The action will be assumed to be unital whenever $Q$ is. The notations for joins, top, bottom, are similar to those of quantales themselves, and the action of an element $a\in Q$ on $x\in M$ is denoted by $ax$. By a \emph{homomorphism} of left $Q$-modules $h:M\to N$ is meant a $Q$-equivariant homomorphism of sup-lattices.

\begin{example}\label{quantaleexamples}
The following are examples of involutive quantales:
\begin{enumerate}
\item Any frame $L$ is a unital involutive quantale with $e=1$, trivial involution and multiplication $ab=a\land b$.
\item Let $G$ be a topological groupoid.
If $G$ is open the topology $\topology(G_1)$ is an involutive quantale with the product of two open sets $U$ and $V$ being given by the pointwise multiplication of groupoid arrows:
\[UV=m(U\times_{G_0} V)=\{m(x,y)\st x\in U,\ y\in V,\ r(x)=d(y)\}\;;\]
and the involute of an open set $U$ is its pointwise inverse $U^*=i(U)$. The quantale is unital if and only if $G$ is an \'etale groupoid, in which case the unit $e$ coincides with the set of unit arrows $u(G_0)$.
Analogous facts apply to localic groupoids (see section \ref{sec:supp}).
\item In particular, the topology of any topological group is an involutive quantale, and the quantale is unital if and only if the group is discrete.
\item Another particular example is $\pwset{X\times X}$, the quantale of binary relations on the set $X$ \cite{MP0}, which is the discrete topology of the greatest equivalence relation on $X$, sometimes referred to as the ``pair groupoid'' on $X$.
\item Let $A$ be a C*-algebra. The set of closed (under the norm topology) linear subspaces of $A$ is an involutive quantale $\Max A$ \cite{Curacao,Mulvey-enc,MP1}. The multiplication of two closed linear subspaces is the closure of the linear span of their pointwise product. The involute of a closed linear subspace if its pointwise involute. The quantale is unital if $A$ has a unit $\mathbf 1$, in which case $e$ is the linear span $\CC \mathbf 1$.
\end{enumerate}
\end{example}

\paragraph{Involutive quantaloids.} Quantaloids are the many objects generalization of quantales. However, although at odds with our terminology for involutive quantales, we shall not need to consider quantaloids without units:

\begin{definition}\label{def:quantaloid}
\begin{enumerate}
\item By a \emph{quantaloid} is meant a sup-lattice enriched category.
\item An \emph{involutive quantaloid} is a quantaloid $\mathcal Q$ equipped with a contravariant sup-lattice enriched isomorphism $(-)^*:\opp{\mathcal Q}\to\mathcal Q$ which is both the identity on objects and its own inverse.
\item An involutive quantaloid $\mathcal Q$ is \emph{modular} if any arrows $a,b,c\in\mathcal Q$ satisfy the \emph{modularity axiom} of Freyd whenever the compositions are defined:
\begin{eq}\label{freydmod}
ab\wedge c\le a(b\wedge a^* c)
\end{eq}%
\end{enumerate}
\end{definition}

The category of sets with relations as morphisms, $\Rel$, is the prototypical example of a modular quantaloid: the morphisms $R:X\to Y$ are the relations $R\subset Y\times X$; and the (total) functions $f:X\to Y$ can be identified with the relations $R:X\to Y$ such that $RR^*\subset\Delta_Y$ and $R^*R\supset\Delta_X$; that is, $R^*$ is right adjoint to $R$ (equivalently, $R$ has a right adjoint, which is necessarily $R^*$). This justifies the following notation and terminology:

\begin{definition}\label{mapsininvqs}
Let $\mathcal Q$ be an involutive quantaloid.
\begin{enumerate}
\item A \emph{map} is a morphism $f:x\to y$ which is left adjoint to $f^*$; that is, $ff^*\le\ident_y$ and $f^*f\ge\ident_x$. The map $f$ is \emph{injective} if it further satisfies $f^*f=\ident_x$, and \emph{surjective} if it satisfies $ff^*=\ident_y$.
If the map $f$ is both injective and surjective we say that it is \emph{unitary}.
\item The subcategory of $\mathcal Q$ containing the same objects as $\mathcal Q$ and the maps as morphisms is denoted by $\maps(\mathcal Q)$.
\end{enumerate}
\end{definition}

\begin{remark}
It is more or less standard, in the case of non-involutive quantaloids, to use the terminology \emph{map} for a morphism that has a right adjoint. This does not coincide, in general, with the above definition in the case of an involutive quantaloid, but for modular quantaloids the definitions coincide \cite[Th.\ 2.2]{Garraway}.
\end{remark}

\begin{definition}\label{quantaloidequiv}
Two involutive quantaloids $\mathcal Q$ and $\mathcal R$ are \emph{equivalent} if there exist two sup-lattice enriched and involution preserving functors (in other words, two \emph{homomorphisms} of involutive quantaloids)
\[
\xymatrix{
\mathcal Q\ar@/^/[r]^F&\mathcal R\ar@/^/[l]^G
}
\]
such that $G\circ F$ and $F\circ G$ are naturally isomorphic to $\ident_{\mathcal Q}$ and $\ident_{\mathcal R}$, respectively, via \emph{unitary} natural isomorphisms (\ie, natural isomorphisms whose components are unitary maps).
\end{definition}

It is an easy exercise to show that an adjoint equivalence (of categories) $F\dashv G$ between the involutive quantaloids $\mathcal Q$ and $\mathcal R$ is an equivalence in the stronger sense just defined if $G$ is a homomorphism of involutive quantaloids and the unit of the adjunction is unitary (equivalently, $F$ is a homomorphism and the counit is unitary).

\subsection{Supported quantales}\label{sec:supp}

The quantales which are associated to \'etale groupoids are the inverse quantal frames. They are instances of the more general and algebraically well behaved class of stable quantal frames, which in turn is included in the equally well behaved class of stably supported quantales. We begin by recalling some properties of these quantales, following \cite{Re07}, and we study their relation to modularity.

\paragraph{Groupoid quantales.}

Let us recall a few aspects of the correspondence between localic groupoids and quantales. Let $G$ be an open localic groupoid.
Since the multiplication map $m$ is open, there is a sup-lattice homomorphism defined as the following composition (in $\SL$):
\[\xymatrix{\opens(G_1) \otimes(G_1)\ar@{->>}[r]& \opens(G_2)\ar[r]^-{m_!}& \opens(G_1)}\;.\]
This defines an associative multiplication on $\opens(G_1)$ which together with the isomorphism \[\opens(G_1)\stackrel{i_!}\to \opens(G_1)\] makes $\opens(G_1)$ an involutive quantale. This quantale is denoted by $\opens(G)$ --- it is the ``opens of $G$''.

The involutive quantale $\opens(G)$ of an open groupoid $G$ is unital if and only if $G$ is \'etale \cite[Corollary 5.12]{Re07}, in which case the unit is $e=u_!(1)$ and $u_!$ defines an order-isomorphism
\[u_!:\opens(G_0)\stackrel\cong\longrightarrow
\downsegment(e)=\{a\in \opens(G)\st a\le e\}\;.\]
Hence, in particular, $\downsegment(e)$ is a frame (\cf\ \ref{def:baselocale}).

\paragraph{Stably supported quantales.} Let $Q$ be a unital involutive quantale.
We recall that by a \emph{support} on $Q$ is meant a sup-lattice homomorphism $\spp:Q\to Q$ satisfying the following conditions for all $a\in Q$:
\begin{eqarray}
\spp(a)&\le& e\label{spp1}\\
\spp(a)&\le& aa^*   \label{spp2}\\
a&\le&\spp(a)a\;. \label{spp3}
\end{eqarray}%
The support is said to be \emph{stable}, and the quantale is \emph{stably supported}, if in addition we have, for all $a,b\in Q$:
\begin{eq}\label{spp4}
\spp(ab)=\spp(a\spp(b))\;.
\end{eq}%
\begin{example}
The quantale $\opens(G)$ of an \'etale groupoid $G$ is stably supported, and the support is given by
$\spp=u_!\circ d_!:\opens(G)\to\opens(G)$ (\cf\ proof of \cite[Theorem 5.11]{Re07}).
\end{example}

For any quantale $Q$ with a support, the following equalities hold for all $a,b\in Q$ \cite[Lemma 3.3(12)]{Re07},
\begin{eqarray}
\spp(a)1 &=& a1\;,\label{spp5}\\
\spp(b)&=&b\ \ \ \ \textrm{if }b\le e\;,\label{spp5new}
\end{eqarray}%
and the unital involutive subquantale $\downsegment(e)=\{a\in Q\st a\le e\}$ is a base locale in the following sense \cite[Lemma 3.3]{Re07}:

\begin{definition}\label{def:baselocale}
Let $Q$ be a unital involutive quantale and let $B=\downsegment(e)$. We say that $B$ is a \emph{base locale} for $Q$ if $b=b^*$ and $bc=b\land c$ for all $b,c\in B$. ($B$ is necessarily a frame, by \cite[\S III.1]{JT}.)
\end{definition}

We further recall \cite[Lemma 3.4 and Theorem 3.8]{Re07} that any stably supported quantale $Q$ admits a unique support, which is given by the following formulas,
\begin{eqarray}
\spp(a) &=& a1\land e\:,\label{spp6}\\
\spp(a) &=& aa^*\land e\;,\label{spp7}
\end{eqarray}%
and, moreover, a support is stable if and only if
\begin{eq}\label{spp8}
\spp(a1)\le\spp(a)
\end{eq}%
for all $a\in Q$.

It has also been proved \cite[Lemma 3.4-5]{Re07} that, if $Q$ is a stably supported quantale, then
\begin{eq}\label{spp9}
ba=b1\wedge a
\end{eq}%
for all $b\in B$ and $a\in Q$, from which the following useful property follows:

\begin{lemma}\label{prop:sqf}
Let $Q$ be a stably supported quantale, and let $a,b\in Q$ with $b\le e$. Then
$ba\wedge e = b\wedge a$.
\end{lemma}

\begin{proof}
$ba\wedge e = (b1\wedge a)\wedge e = (b1\wedge e)\wedge a=be\wedge a=b\wedge a$. \qed
\end{proof}

Furthermore, every stable support is $B$-equivariant, because for all $a\in Q$ and $b\in B$ we have, by (\ref{spp5new}), $\spp(ba)=\spp(b\spp(a))=b\spp(a)$,
and in fact equivariance is equivalent to stability (\cf\ \ref{previouslyunnoticed}):

\begin{theorem}\label{thm:stabiffequiv}
Let $Q$ be a supported quantale.
The support of $Q$ is stable if and only if it is a homomorphism of $B$-modules.
\end{theorem}

\begin{remark}\label{rem:weakspp}
All the above properties of supports and stable supports still hold if the definition of support is weakened by requiring supports to be only monotone instead of join-preserving. In particular, \ref{thm:stabiffequiv} still holds because if $\spp:Q\to B$ is a monotone $B$-equivariant map satisfying (\ref{spp2})--(\ref{spp3}) then it is left adjoint to the assignment $(-)1:B\to Q$, and hence it preserves joins. Hence, the exact definition of support is irrelevant as far as stable supports are concerned.
\end{remark}

\paragraph{Inverse quantal frames.} By a \emph{stable quantal frame} is meant a stably supported quantale which is also a frame. The following condition holds for all stable quantal frames \cite[Lemma 4.17(28)]{Re07} and will be used in the proof of \ref{thm:Qsheaf}:
\begin{eq}\label{sqfprop}
(a\wedge e)1 \ge \V_{yz^*\le a} y\wedge z\;.
\end{eq}%

An \emph{inverse quantal frame} is a stable quantal frame $Q$ that satisfies the \emph{cover condition}
\[\V\ipi(Q)=1\;,\]
where $\ipi(Q)=\{s\in Q\st ss^*\vee s^* s\le e\}$ is the set of \emph{partial units} of $Q$. This set, equipped with the multiplication of $Q$, has the structure of a complete and infinitely distributive inverse monoid whose inverses are given by the involution of the quantale \cite[Corollary 3.26]{Re07}. We remark that we have
\begin{eq}\label{sppIQ}
\spp(s)=ss^*\ \ \ \ \textrm{for all }s\in\ipi(Q)\;.
\end{eq}%

The inverse quantal frames $Q$ are precisely the quantales of the form $Q\cong\opens(G)$ for a localic \'etale groupoid $G$ \cite[Theorem 4.19 and Theorem 5.11]{Re07}.

\begin{example}\label{exmtopologicalgroupoid}
For the sake of illustration, let us describe this correspondence in the case of a topological \'etale groupoid $G$. The topology $\topology(G_1)$ is an inverse quantal frame (\cf\ \ref{quantaleexamples}) whose support is given by $\spp(U)=u(d(U))$ for all open sets $U$ of $G_1$.
By a \emph{local bisection} of $G$ is meant a continuous local section $s$ of $d$ such that $r\circ s$ is an open embedding of the domain of $s$ into $G_0$, and the partial units are precisely the images of the local bisections. Equivalently, a partial unit is the same as an open set $U\in\topology(G_1)$ such that the restrictions $d\vert_U$ and $r\vert_U$ are injective. In particular, the partial units of the quantale of binary relations $\pwset{X\times X}$ on a set $X$ are the partial bijections on $X$.
\end{example}

\paragraph{Modular quantales.}

The notion of modularity of Freyd (\cf\ \ref{def:quantaloid}) is crucial in his characterization of abstract quantaloids of binary relations. Similarly, the existence of stable supports provides us with a definition of what may be meant by an abstract quantale of binary relations, as in \cite{MarcR}. We are thus provided with two natural ways of abstracting quantales of binary relations, and it is worth comparing them. In addition, the fact that inverse quantal frames are modular (\cf\ \ref{thm:iqfaremod}) will play a role at the end of section \ref{sec:backgrpdshvs}.

As a first step we see that stably supported quantales are more general than modular quantales:

\begin{theorem}
Every modular quantale is stably supported.
\end{theorem}

\begin{proof}
An involutive quantale $Q$ is modular (\cf\ \ref{def:quantaloid}) if for all $a,b,c\in Q$ we have
\begin{eq}\label{modularity1}
a(a^*b\wedge c)\ge b\wedge ac
\end{eq}%
or, equivalently, for all $a,b,c\in Q$ we have
\begin{eq}\label{modularity2}
(c\wedge ba^*)a\ge ca\wedge b\;.
\end{eq}%
Let then $Q$ be modular, and define the operation $\spp:Q\to B$ by
\[\spp(a)=aa^*\wedge e\;.\]
This operation is monotone, and in order to see that it is a stable support we check that it satisfies the required three laws, namely (\ref{spp2})--(\ref{spp3}) and stability. Whereas (\ref{spp2}) holds almost by definition, (\ref{spp3}) follows from a direct application of (\ref{modularity2}):
\[\spp(a)a=(aa^*\wedge e)a\ge a\wedge ea=a\;.\]
And we obtain stability by a direct application of (\ref{modularity1}):
\[\spp(a1)=a1a^*\wedge e=a1a^*\wedge e\wedge e\le a(1a^*\wedge a^*)\wedge e=aa^*\wedge e=\spp(a)\;.\qed\]
\end{proof}

The two notions do not coincide, however, as the following example due to Jeff Egger shows:

\begin{example}\label{exm:ssqnotmod}
Let $Q$ be the 8-element boolean algebra with atoms $a$, $b$, $c$, equipped with the trivial involution and the following multiplication:
\[
\begin{array}{c|cccccccc}
   & 0 & a & b & c & x & y & z & 1\\
\hline0 & 0 & 0 & 0 & 0 & 0 & 0 & 0 & 0\\
a & 0 & a & b & c & x & y & z & 1\\
b & 0 & b & 1 & 1 & 1 & 1 & 1 & 1\\
c & 0 & c & 1 & 1 & 1 & 1 & 1 & 1\\
x & 0 & x & 1 & 1 & 1 & 1 & 1 & 1\\
y & 0 & y & 1 & 1 & 1 & 1 & 1 & 1\\
z & 0 & z & 1 & 1 & 1 & 1 & 1 & 1\\
1 & 0 & 1 & 1 & 1 & 1 & 1 & 1 & 1
\end{array}
\]
Defining $\spp(q)=a$ for all $q\neq 0$ we obtain a stable support, but $Q$ is not modular because 
$bc\wedge a = 1\wedge a = a$ and $b(c\wedge ba) = b(c\wedge b) = b0 = 0$.
\end{example}

In this example the quantale is also a frame. Hence, modularity is stronger than being stably supported even for quantal frames. In turn, every inverse quantal frame is necessarily modular, as has been mentioned in \cite{HS2} by taking into account the representation of inverse quantal frames by inverse semigroups of \cite{Re07}. A direct proof is the following:

\begin{theorem}\label{thm:iqfaremod}
Every inverse quantal frame is modular, but not every modular quantal frame is an inverse quantal frame.
\end{theorem}

\begin{proof}
Let $Q$ be an inverse quantal frame, and let $a,b,c\in Q$.
Let $a=\V_i s_i$, $b=\V_j t_j$ and $c=\V_k u_k$, where $s_i$, $t_j$ and $u_k$ are partial units for all $i,j,k$. We have
\[a(b\wedge a^*c)=\V_{i,j,k,\ell} s_i(t_j\wedge s_\ell^* u_k)\ge 
\V_{i,j,k} s_i(t_j\wedge s_i^* u_k)=\V_{i,j,k} s_i t_j\wedge s_i s_i^* u_k\]
and
\[
ab\wedge c= \V_{i,j,k} s_i t_j\wedge u_k\;,
\]
and modularity follows from the equality
$s_i t_j\wedge s_i s_i^* u_k=s_i t_j\wedge u_k$. An example showing that not every modular quantal frame is an inverse quantal frame is the four-element quantale $R$ of \ref{exm:hilbsec3}. \qed
\end{proof}

\subsection{Locale sheaves as modules}\label{sec:localesheavesasmodules}

Sheaves on locales can be described as quantale modules (on locales) in more than one way. With the exception of \ref{lem:firsthilbertsections}, all the statements that follow are recalled from \cite{RR}, whose terminology we follow.

\paragraph{Maps as modules.}
If $p:X\to B$ is a map of locales then $\opens(X)$ is an $\opens(B)$-module by change of ``base ring'' along $p^*$; that is, the action is given by
\[bx = p^*(b)\wedge x\]
for all $b\in \opens(B)$ and $x\in \opens(X)$. This makes $X$ an \emph{$\opens(B)$-locale}, by which is meant that $\opens(X)$ is equipped with a structure of $\opens(B)$-module satisfying the condition $b1\land x=bx$ for all $b\in\opens(B)$ and $x\in\opens(X)$. We remark that the map $p$ can be recovered from the module structure by the condition $p^*(b)=b1$.

By a \emph{map of $\opens(B)$-locales} is meant a map of locales $f$ whose inverse image $f^*$ is a homomorphism of $\opens(B)$-modules. The resulting \emph{category of $\opens(B)$-locales} is denoted by $\opens(B)$-$\Loc$, and it is isomorphic to the slice category $\Loc/B$ \cite[Theorem 1]{RR}.

\paragraph{Open maps.}
If $p:X\to B$ is an open map the unit of the adjunction $p_!\dashv p^*$ gives us $p_!(x)x=x$ for all $x\in \opens(X)$. Conversely, if $X$ is a locale for which $\opens(X)$ is an $\opens(B)$-module equipped with a homomorphism $\spp:\opens(X)\to\opens(B)$ of $\opens(B)$-modules such that $\spp(x)x=x$ for all $x\in\opens(X)$ then $X$ is an $\opens(B)$-locale and the corresponding map of locales $p:X\to B$ is open with $p_!=\spp$ \cite[Theorem 3]{RR}.

Such an $\opens(B)$-locale is called \emph{open}. For each $x\in\opens(X)$ the element $\spp(x)$ is referred to as the \emph{support} of $x$, and $\spp$ itself is called the \emph{support} of $X$, in imitation of the terminology for supported quantales (\cf\ section \ref{sec:supp}).

\paragraph{Sheaves.}
Now let $p:X\to B$ be a local homeomorphism. The images of the local sections of $p$ can be identified \cite[\S 2.3]{RR} with the elements $s\in\opens(X)$ such that
\begin{eq}\label{eq:locsec1}
\forall_{x\in \opens(X)}\ \ x\le s\Rightarrow x=\spp(x) s\;.
\end{eq}%
Henceforth we shall refer to the elements that satisfy (\ref{eq:locsec1}) simply as \emph{local sections}, and we shall denote the set of all the local sections by $\sections_X$. Of course, we have $\V\sections_X=1$ (`the local sections cover $X$').

Any open $\opens(B)$-locale $X$ which is thus covered by the local sections is called an \emph{\'etale $\opens(B)$-locale} and the full subcategory of $\opens(B)$-$\Loc$ whose objects are the \'etale $\opens(B)$-locales is denoted by $\opens(B)$-$\Etale$. Of course, $\opens(B)$-$\Etale$ is equivalent to $\Etale/B$, the full subcategory of $\Loc/B$ whose objects are the local homeomorphisms into $B$, which in turn is equivalent to $\sh(B)$, the category of sheaves on $B$ and natural transformations between them. Hence, from here on we adopt the following shorter terminology:

\begin{definition}\label{sec:Asheaf}
Let $A$ be a frame. By an \emph{$A$-sheaf} is meant an \'etale $A$-locale.
\end{definition}

If $X$ and $Y$ are $\opens(B)$-sheaves, by a \emph{sheaf homomorphism} \[h:\opens(X)\to\opens(Y)\] is meant a homomorphism of $\opens(B)$-modules which preserves supports and local sections; that is,
\begin{eqarray}
\spp(h(x))&=&\spp(x)\textrm{ for all }x\in \opens(X)\\
h(\sections_X)&\subset&\sections_Y\;.
\end{eqarray}%
The sheaf homomorphisms are the direct images $f_!:\opens(X)\to\opens(Y)$ of the maps $f:X\to Y$ \cite[Theorem 5]{RR}. The category whose objects are the $\opens(B)$-sheaves and whose arrows are the sheaf homomorphisms between them is isomorphic to $\opens(B)$-$\Etale$ and it is denoted by $\opens(B)$-$\sh$.

There is an alternative way of describing the sheaves on $B$, in terms of Hilbert modules on $\opens(B)$ (\cf\ section \ref{sec:hilbertmodules}): the $\opens(B)$-sheaves are precisely the same as the Hilbert $\opens(B)$-modules which are equipped with Hilbert bases (\cf\ \ref{exm:BsheavesasHilbertmodules}). The Hilbert module inner product of an $\opens(B)$-sheaf $X$ is given by
\begin{eq}\label{firstip}
\inner x y = \spp(x\land y)\;,
\end{eq}%
and the adjoint $\varphi^\dagger$ of a sheaf homomorphism $\varphi=f_!:\opens(X)\to\opens(Y)$, which is defined by the condition $\inner{\varphi(x)}y=\inner x{\varphi^\dagger(y)}$, coincides with the inverse image homomorphism $f^*$ \cite[Theorem 11]{RR}.

We conclude this brief exposition on locale sheaves with a useful fact not mentioned in \cite{RR}:

\begin{lemma}\label{lem:firsthilbertsections}
Let $B$ be a locale, let $X$ be an $\opens(B)$-sheaf, and let $s\in\opens(X)$. Then $s$ is a local section if and only if
\begin{eq}\label{eq:locsec2}
\forall_{x\in \opens(X)}\ \ \inner x s s\le x\;.
\end{eq}%
\end{lemma}

\begin{proof}
The equivalence is easily proved:
\begin{itemize}
\item If $s$ satisfies (\ref{eq:locsec1}) and $x\in \opens(X)$ then $x\land s\le s$ and, hence, we have
$\inner x s s=\spp(x\land s)s=x\land s\le x$.
\item Conversely, if $s$ satisfies (\ref{eq:locsec2}) and $x\le s$ then
$x=\spp(x) x=\spp( x\land x) x\le\spp(x\land s)s=\inner x s s\le x$. \qed
\end{itemize}
\end{proof}

\section{Groupoid actions as quantale modules}\label{sec:gaqm}

In this section we show that the assignment from open groupoids to quantales has a one-sided generalization whereby actions of open groupoids define quantale modules. We shall begin by addressing the more general situation, for open groupoids, after which \'etale groupoids will be considered along with actions on open maps and local homeomorphisms. A module-theoretic formulation of the actions of \'etale groupoids will be obtained, and a first description of groupoid sheaves in terms of quantale modules will be achieved.

\subsection{Actions of open groupoids}

\paragraph{Preliminaries on groupoid actions.}

Let $G$ be a localic groupoid. By a \emph{locale over $G_0$}, or simply a \emph{$G_0$-locale}, will be meant a locale $X$ together with a map $p:X\to G_0$ called the \emph{projection} into $G_0$. The \emph{category of $G_0$-locales} is the slice category $\Loc/G_0$. A \emph{(left) action} of $G$ on the $G_0$-locale $(X,p)$ is a map of locales
$\mathfrak a:G_1\pb X\to X$
such that the following diagrams commute, where $G_1\pb X$, $G_2\pb X$ and $G_1\pb(G_1\pb X)$ are pullbacks in $\Loc$ respectively of $r$ and $p$, $r\circ\pi_2$ and $p$, and $r$ and $d\circ\pi_1$:
\begin{eqarray}
&\vcenter{\xymatrix{G_1\pb X\ar[rr]^-{\pi_1}\ar[d]_{\mathfrak a}&&G_1\ar[d]^d\\
X\ar[rr]_{p}&&G_0}} &  \label{crucialpb}
\end{eqarray}%
\begin{eqarray}
&\vcenter{\xymatrix{
G_1\pb(G_1\pb X)\ar[d]_{\cong}\ar[rrr]^-{1\times\mathfrak a}&&&G_1\pb X\ar[dd]^{\mathfrak a}\\
G_2\pb X\ar[d]_{m\times 1}\\
G_1\pb X\ar[rrr]_-{\mathfrak a}&&&X
}} & \textrm{(Associativity)}\label{associativity}
\end{eqarray}%
\begin{eqarray}
&\vcenter{\xymatrix{&G_1\pb X\ar[dr]^{\mathfrak a}\ar@{<-}[dl]_{\langle u\circ p,1\rangle}\\
X\ar@{=}[rr]&&X
}} & \textrm{(Unitarity)}\label{unitality}
\end{eqarray}%
The $G_0$-locale $(X,p)$ together with the action $\mathfrak a$ will be referred to as a \emph{(left) $G$-locale} and we shall denote it by $(X,p,\mathfrak a)$, or simply by $X$ when no confusion will arise.

The following simple fact will be useful a few times later on:

\begin{lemma}\label{crucialemma}
Let $p:X\to G_0$ be a map of locales and let $G_1\pb X$ be the pullback of $r$ and $p$. Then the projection
$\pi_1:G_1\pb X\to G_1$ coincides with the map
$m\circ (1\times(u\circ p))$. In particular, (\ref{crucialpb}) is equivalent to the equation
$p\circ\mathfrak a = d\circ m\circ (1\times(u\circ p))$.
\end{lemma}

\begin{proof}
This follows from the commutativity of the following diagram, whose left triangle is obviously commutative and whose right triangle is commutative due to one of the unit laws of $G$:
\[
\vcenter{\xymatrix{
G_1\pb X\ar[rr]^{1\times p}\ar[rrd]_{\pi_1}
&&G_1\pb G_0\ar[rr]^{1\times u}\ar[d]^{\pi_1}_{\cong}
&&G_1\pb G_1\ar[lld]^m\\
&&G_1
}}
\qed\]
\end{proof}

\paragraph{From actions to modules.} It is easy to show that the diagram (\ref{crucialpb}) is a pullback (briefly, because the action can be reversed due to the inversion operation $i$ of the groupoid), and thus if $G$ is an open groupoid the action map $\mathfrak a$ is necessarily open. Hence, in this case, taking into account that $G_1\pb X$ is, in $\Frm$, a quotient $G_1\tensover X$ of the tensor product $G_1\otimes X$, we obtain a sup-lattice homomorphism by composing with the direct image of the action:
\[\xymatrix{G_1\otimes X\ar@{->>}[r]&G_1\tensover X\ar[r]^-{\mathfrak a_!}&X}\]
Showing that this defines an action of $\opens(G)$ on $X$ (a left quantale module) is straightforward and essentially the same as the proofs of associativity and unit laws for the quantale $\opens(G)$ (\cf\ \cite[Theorem 5.2 and Lemma 5.8]{Re07}).

\begin{definition}
Let $G$ be an open groupoid. We shall denote by $\opens(X)$ the left $\opens(G)$-module which is obtained from a $G$-locale $X$.
\end{definition}

\paragraph{Equivariant maps.} Let $X$ and $Y$ be $G$-locales with actions $\mathfrak a$ and $\mathfrak b$, respectively. An \emph{equivariant map} from $X$ to $Y$ is a map $f:X\to Y$ in $\Loc/G_0$ that commutes with the actions; that is, such that the following diagram commutes:
\[\xymatrix{G_1\pb X\ar[rrr]^-{1\times f}\ar[d]_{\mathfrak a}&&&G_1\pb Y\ar[d]^{\mathfrak b}\\
X\ar[rrr]_f&&&Y}
\]
We shall refer to the category of $G$-locales and equivariant maps between them as $G$-$\Loc$.
It is simple to see that, since $G$ is a groupoid rather than just a category, the above diagram is actually a pullback. Hence, if $G$ is an open groupoid, in which case as we have seen the actions are open maps, the following diagram in $\SL$ also commutes \cite[Proposition V.4.1]{JT}:
\[\xymatrix{G_1\tensover X\ar@{<-}[rrr]^-{1\otimes f^*}\ar[d]_{\mathfrak a_!}&&&G_1 \tensover Y\ar[d]^{\mathfrak b_!}\\
X\ar@{<-}[rrr]_{f^*}&&&Y}
\]
This implies that the locale homomorphism $f^*$ commutes with the actions of $\opens(G)$ on $\opens(X)$ and $\opens(Y)$, and thus it is a homomorphism of $\opens(G)$-modules. Hence,
we obtain:

\begin{lemma}\label{thm:firstfunctor}
The assignments $X\mapsto\opens(X)$ and $f\mapsto f^*$ define a faithful functor $\opens:G\textrm{-}\Loc\to\opp{\opens(G)\textrm{-}\Mod}$.
\end{lemma}
Comparing this with \cite[Theorem 5.14 and Example 5.15]{Re07} we see that the assignment from groupoid actions to modules has better functorial properties than the assignment from groupoids to quantales.

This functor is not full, but we make the following observation:

\begin{lemma}\label{thm:lax}
Let $G$ be an open groupoid and let $f:X\to Y$ be a map of locales such that $f^*$ is a homomorphism of $\opens(G)$-modules. Denoting the actions of $X$ and $Y$ by $\mathfrak a$ and $\mathfrak b$, respectively, we have $f\circ\mathfrak a\ge\mathfrak b\circ(1\times f)$.
\end{lemma}

\begin{proof}
Let us prove the inverse image version of the inequality, that is
\[\mathfrak a^*\circ f^*\ge \mathfrak (1\otimes f^*)\circ b^*\;,\] using the equality
$f^*\circ\mathfrak b_!=\mathfrak a_!\circ(1\otimes f^*)$ that corresponds to the $Q$-equivariance of $f^*$:
\[\mathfrak a^*\circ f^*\ge \mathfrak a^*\circ f^*\circ \mathfrak b_!\circ\mathfrak b^*=\mathfrak a^*\circ\mathfrak a_!\circ (1\otimes f^*)\circ\mathfrak b^*\ge (1\otimes f^*)\circ\mathfrak b^*=(1\times f)^*\circ\mathfrak b^*\;. \qed\]
\end{proof}

\subsection{Actions of \'etale groupoids}\label{sec:actions}

Now we study actions of localic \'etale groupoids. As we shall see, the existence of a base locale (\cf\ \ref{def:baselocale}) for the quantale of such a groupoid enables us to extend to groupoid actions the module language of locale sheaves (\cf\ \ref{sec:localesheavesasmodules}). We remark that more could have been said along these lines for open groupoids, too, since there is a (more general) notion of base locale for the quantales of these \cite{ProtinPhD,protinresende}, but for the purposes of this paper that is not needed.

\paragraph{$Q$-locales.}

For any localic \'etale groupoid $G$, if $X$ is a $G$-locale with projection $p:X\to G_0$ then $\opens(X)$ is an $\opens(G_0)$-module by change of ``ring'' along the inverse image homomorphism $p^*:\opens(G_0)\to \opens(X)$. Letting $B$ denote the base locale of $\opens(G)$, the same action of $\opens(G_0)$ on $\opens(X)$ can be obtained through the isomorphism $\opens(G_0)\cong B$ by restricting the action of $Q$:

\begin{lemma}\label{lem:actionscoincide}
Let $G$ be an \'etale groupoid and let $X$ be a $G$-locale with projection $p:X\to G_0$. For all $b\in \opens(G_0)$ and $x\in \opens(X)$ we have $u_!(b)x=p^*(b)\land x$. In particular, $\opens(X)$ is a unital $Q$-module and the action uniquely defines $p$ by the equation $p^*(b)=u_!(b)1$.
\end{lemma}

\begin{proof}
Axiom (\ref{unitality}) of $G$-locales is $\mathfrak a\circ \langle u\circ p,1\rangle=1$,
which we can rewrite as $\mathfrak a\circ(u\times 1)\circ\langle p,1\rangle=1$, where  the pairing $\langle p,1\rangle:X\to G_0\pb X$ is an isomorphism and thus
$\mathfrak a\circ(u\times 1) = \langle p,1\rangle^{-1}$.
Hence, we have \[\mathfrak a_!\circ(u_!\otimes 1)=[p^*,1]\]
and the required equation follows:
\begin{eqnarray*}
u_!(b)x &=& \mathfrak a_!(u_!(b)\otimes x)=(\mathfrak a_!\circ(u_!\otimes 1))(b\otimes x)\\
&=&[p^*,1](b\otimes x) = p^*(b)\land x\;. \qed
\end{eqnarray*}
\end{proof}

Hence, the faithful functor $\opens:G$-$\Loc\to \opp{Q\textrm{-}\Mod}$ of \ref{thm:firstfunctor} restricts to a functor to the following category $Q$-$\Loc$:

\begin{definition}
Let $Q$ be an inverse quantal frame with base locale $B$. By a \emph{$Q$-locale} will be meant a locale $X$ such that $\opens(X)$ is a (unital) left $Q$-module whose action satisfies the condition $bx=b1\land x$ for all $b\in B$ and $x\in \opens(X)$. The \emph{category of $Q$-locales}, $Q$-$\Loc$, is that whose objects are the $Q$-locales and whose morphisms $f:X\to Y$ are the maps of locales such that $f^*$ is a homomorphism of $Q$-modules.
\end{definition}

\begin{example}\label{QisQloc}
Any inverse quantal frame $Q$ itself defines a $Q$-locale, since $(G,d,m)$ is a $G$-locale: the equality $ba=b1\land a$ holds for all $b\in B$ and $a\in Q$, and, due to the involution, $ab=1b\land a$ also holds (corresponding to the right $G$-locale structure of $G$ with projection $r$).
\end{example}

\begin{example}\label{QtensorX}
Let $Q=\opens(G)$ be an inverse quantal frame with base locale $B$. If $X$ is a $B$-locale then $Q\otimes_B \opens(X)$ is a frame whose natural left $Q$-action defines a $Q$-locale:
\[b(a\otimes x)=ba\otimes x=(b1\land a)\otimes x=b(1\otimes 1)\land (a\otimes x)\;.\]
If $X$ corresponds to a $G_0$-locale $p:X\to G_0$ then the $Q$-locale $Q\otimes_B \opens(X)$ corresponds to a $G$-locale $G_1\pb X$ whose projection $d\circ\pi_1$ (where $\pi_1$ is the pullback of $p$ along $r$) is an open map (resp.\ a local homeomorphism) if $p$ is.
\end{example}

\begin{example}\label{Blocales}
If the inverse quantal frame $Q$ coincides with its base locale $B$ (\ie, the corresponding groupoid $G$ is just the locale $G_1=G_0$ with identity structure maps) the category $B$-$\Loc$ is that of section \ref{sec:localesheavesasmodules}.
\end{example}

\paragraph{Multiplicativity.}

Let $G$ be an \'etale groupoid. Any left $\opens(G)$-module $M$ (not necessarily an $\opens(G)$-locale, or even a locale) is also a left $B$-module due to the inclusion of its base locale $B\subset \opens(G)$. Hence, we can form the tensor product $\opens(G)\otimes_B M$. The associativity of the action $\opens(G)\otimes M\to M$ implies that it factors through the quotient $\opens(G)\otimes M\to \opens(G)\otimes_B M$ and a sup-lattice homomorphism $\alpha:\opens(G)\otimes_B M\to M$, whose right adjoint $\alpha_*$ is given by
\begin{eqarray}
\alpha_*(x) &=& \V\{a\otimes y\in \opens(G)\otimes_B M\st \alpha(a\otimes y)\le x\}\\
 &=& \V\{a\otimes y\in \opens(G)\otimes_B M\st ay\le x\}\label{rightadjoint}\;.
\end{eqarray}%
But the fact that $\opens(G)$ is an inverse quantal frame provides us with a more useful formula for $\alpha_*$:

\begin{lemma}\label{alphastar}
Let $Q$ be an inverse quantal frame with base locale $B$ and let $M$ be a left $Q$-module with action $\alpha:Q\otimes_B M\to M$. The right adjoint $\alpha_*$ is given by, for all $x\in M$,
\begin{eq}\alpha_*(x)=\V_{s\in\ipi(Q)} s\otimes s^*x\;.\label{alphastareq}
\end{eq}%
It follows that $\alpha_*$ preserves arbitrary joins (besides arbitrary meets).\end{lemma}

\begin{proof}
Since $\ipi(Q)$ is join-dense in $Q$ and joins distribute over tensors we can equivalently replace $a$ in (\ref{rightadjoint}) by $s\in\ipi(Q)$ and thus obtain
\begin{eqnarray*}
\alpha_*(x) &=& \V_{sy\le x} s\otimes y\ \le
\V_{s^*sy\le s^*x} s\otimes y\ =\V_{s^*sy\le s^*x} ss^*s\otimes y\\
&=&
\V_{s^*sy\le s^*x} s\otimes s^*sy\ \ [\textrm{because }s^*s\in B \textrm{ --- } \cf\ (\ref{sppIQ})]\\
&\le&\V_{s\in\ipi(Q)} s\otimes s^*x\ \le\ \alpha_*(x)\;,
\end{eqnarray*}
where the last inequality is a consequence of the fact that for each $s\in\ipi(Q)$ we have $ss^*x\le x$ and thus $s\otimes s^* x\le\alpha_*(x)$. Hence, all the above inequalities are in fact equalities. The fact that $\alpha_*$ preserves joins is an immediate consequence, for if $Y\subset M$ then
\[\alpha_*\left(\V Y\right) = \V_{s\in\ipi(Q)} s\otimes s^*\V Y = \V_{x\in Y}\V_{s\in\ipi(Q)} s\otimes s^*x=\V\alpha_*(Y)\;. \qed\]
\end{proof}

\begin{remark}
This result holds under more general assumptions, namely it suffices that $Q$ be a unital involutive quantale containing a join-dense sub-involutive-semigroup $S\subset Q$ such that $ss^*\le e$ and $s\le ss^*s$ (hence, $s=ss^*s$) for all $s\in S$ (notice that $B=\downsegment(e)$ is always a unital involutive subquantale of $Q$ and the same remarks about the tensor product $Q\otimes_B M$ apply). In this more general situation we obtain
\[\alpha_*(x)=\V_{s\in S} s\otimes s^*x\;.\]
Examples of such quantales are the inverse quantales of \cite{Re07} --- the set $\ipi(Q)$ of partial units of an inverse quantale $Q$ is a join-dense complete inverse monoid whose locale of idempotents coincides with $B$. Such a quantale is of the form $\opens(G)$ for an \'etale groupoid $G$ if and only if it is also a frame \cite{Re07}. As a corollary of this we conclude that the multiplication $\mu:Q\otimes_B Q\to Q$ of an inverse quantale $Q$ necessarily has a join preserving right adjoint given by
\begin{eq}\label{mustar}
\mu_*(a)=\V_{s\in \ipi(Q)} s\otimes s^*a
\end{eq}%
In particular, we obtain in this way a new and simpler proof of the fact that every inverse quantal frame is multiplicative.
\end{remark}

\paragraph{Equivalence between $G$-locales and $Q$-locales.}
Now we shall see that the categories of $G$-locales and of $\opens(G)$-locales, for any \'etale groupoid $G$, amount to the same thing.

\begin{lemma}\label{strictbijection}
Let $G$ be an \'etale groupoid. The assignment $X\mapsto\opens(X)$ from $G$-locales to $\opens(G)$-locales is a (strict) bijection.
\end{lemma}

\begin{proof}
Let $Q=\opens(G)$ and let $X$ be a $Q$-locale. The inclusion of the base locale $B\subset Q$ makes $\opens(X)$ a $B$-locale and thus we have a map $p:X\to G_0$ defined by $p^*(b)=u_!(b)1$ (\cf\ \ref{Blocales}). Since the pullback $G_1\pb X$ of $r$ and $p$ is, in the category of frames, the quotient of the frame coproduct $\opens(G_1)\otimes \opens(X)$ generated by the equalities
\begin{eq}\label{pushout}
\pi_1^*(r^*(b)) = \pi_2^*(p^*(b))\;,
\end{eq}%
the $Q$-locale conditions $p^*(b)\land x=u_!(b)x$ and $a\land r^*(b)=a u_!(b)$ (\cf\ \ref{QisQloc}) show, if we stabilize (\ref{pushout}) under finite meets, that $G_1\pb X$ coincides with the sup-lattice quotient generated by the equalities
$a u_!(b)\otimes x =  a\otimes u_!(b)x$, in other words it is the tensor product of $B$-modules $Q\otimes_B \opens(X)$.
Since the right adjoint $\alpha_*$ of the module action
\[\alpha:Q\otimes_B \opens(X)\to \opens(X)\]
preserves joins (see \ref{alphastar}), we define a groupoid action $\mathfrak a:G_1\pb X\to X$ by $\mathfrak a^*=\alpha_*$ and in order to see that we have obtained a $G$-locale all we need is to verify that the three axioms (\ref{crucialpb})--(\ref{unitality}) are satisfied. Of course, once this is done our proof will be finished because it is clear that the construction of the $G$-locale structure from the $Q$-locale thus obtained is the inverse of the assignment $Y\mapsto\opens(Y)$.

Axiom (\ref{associativity}) (the associativity of $\mathfrak a$) follows in a straightforward manner from the associativity of $\alpha$ because $\alpha=\mathfrak a_!$. (This is completely analogous to the way in which the associativity of the multiplication of an open groupoid follows from the associativity of the multiplication of its quantale, \cf\ \cite[Theorem 4.8]{Re07}.)

Proving the two other axioms is less easy because $p$ is not necessarily an open map and thus we do not have straightforward direct image versions of the axioms we want to prove. Let us start with axiom (\ref{crucialpb}). By \ref{crucialemma}, this is equivalent to the equation $p\circ \mathfrak a=d\circ m\circ (1\times(u\circ p))$, which we can verify directly in terms of inverse images using the formulas (\ref{alphastareq}) and (\ref{mustar}) for $\mathfrak a^*$ and $m^*$: on one hand we have
\[\mathfrak a^*(p^*(b)) = \V_{s\in\ipi(Q)} s\otimes s^*u_!(b)1_X\]
and, on the other,
\[m^*(d^*(b)) = \V_{s\in\ipi(Q)} s\otimes s^*u_!(b)1_Q\;.\]
The inverse image of $1\times(u\circ p)$ is given by
\[(1\otimes(p^*\circ u^*))(a\otimes c)=a\otimes((c\land e)1_X)\]
and, combining these formulas, we obtain
\[1\otimes(p^*\circ u^*)(m^*(d^*(b))) =
 \V_{s\in\ipi(Q)} s\otimes (s^*u_!(b)1_Q\land e)1_X=\mathfrak a^*(p^*(b))\;,\]
where the last step follows from the following three facts: (i) $s^* u_!(b)$ belongs to $\ipi(Q)$; (ii) for all $t\in\ipi(Q)$ we have $t1_Q\land e=\spp(t)=tt^*$; (iii) for all $t\in\ipi(Q)$ we have $tt^* 1_X\le t 1_X = tt^* t 1_X\le tt^* 1_X$, and thus
$(s^* u_!(b)1_Q\land e)1_X=s^* u_!(b) 1_X$.

Now let us verify axiom (\ref{unitality}). The inverse image of $\mathfrak a\circ\langle u\circ p,1\rangle$ is given by
\[[p^*\circ u^*,1](\mathfrak a^*(x)) = \V_{s\in\ipi(Q)} p^*(u^*(s))\land s^*x
= \V_{s\in\ipi(Q)} (s\land e)1_X\land s^* x\;.\]
Since $X$ is a $Q$-locale we have
$(s\land e)1_X\land s^* x=(s\land e)s^* x$ and, since $s$ is in the inverse monoid $\ipi(Q)$, we also have $(s\land e)s^*=s\land e$. Hence,
\[ \V_{s\in\ipi(Q)} (s\land e)1_X\land s^* x = \V_{s\in\ipi(Q)}(s\land e)x=
\left(\V\ipi(Q)\land e\right)x=ex=x\]
and we conclude that $\mathfrak a\circ\langle u\circ p,1\rangle=1$ as required. \qed
\end{proof}

\begin{theorem}\label{lem:fullfaithfullh}
Let $G$ be an \'etale groupoid. The categories $G$-$\Loc$ and $\opens(G)$-$\Loc$ are isomorphic.
\end{theorem}

\begin{proof}
Let $Q=\opens(G)$. All we need to do is show that the functor $\opens:G$-$\Loc\to Q$-$\Loc$ is full.
Let $X$ and $Y$ be $G$-locales, let $f:X\to Y$ be a map of locales such that $f^*$ is a homomorphism of $Q$-modules, and let the actions of $G$ on $X$ and $Y$ be $\mathfrak a$ and $\mathfrak b$, respectively. By \ref{thm:lax}, in order to prove that the functor is full we only have to prove, for all $y\in \opens(Y)$, the inequality
\begin{eq}\label{eq:inequality}
\mathfrak a^*(f^*(y))\le (1\otimes f^*)(\mathfrak b^*(y))\;.
\end{eq}%
From \ref{alphastar} and the fact that $f^*$ is $Q$-equivariant we have
\[\mathfrak a^*(f^*(y))=\V_{s\in\ipi(Q)} s\otimes s^*(f^*(y))=\V_{s\in\ipi(Q)} s\otimes f^*(s^*y)\;.\]
The expression $s\otimes f^*(s^*y)$ on the right equals $(1\otimes f^*)(s\otimes(s^*y))$, and we have $s\otimes(s^*y)\le\mathfrak b^*(y)$ because $\mathfrak b_!(s\otimes(s^*y))=ss^*y\le y$. This proves the inequality (\ref{eq:inequality}).
\qed
\end{proof}

\subsection{Actions on sheaves}\label{sec:actionsonsheaves}

\paragraph{Actions on open maps.}

For any groupoid $G$, by an \emph{open} $G$-locale will be meant a $G$-locale whose projection is an open map. Similarly, for an \'etale groupoid the corresponding $\opens(G)$-locales will be called \emph{open}. Their description is very simple and does not even require the $\opens(G)$-locale condition:

\begin{lemma}
Let $Q$ be an inverse quantal frame with base locale $B$, and let $X$ be a locale such that $\opens(X)$ is a $Q$-module. Then $X$ is an open $Q$-locale if and only if there exists a (necessarily unique) homomorphism of $B$-modules
\[\spp : \opens(X)\to B\]
such that $\spp(x)x=x$ for all $x\in \opens(X)$.
\end{lemma}

\begin{proof}
This follows immediately from the description of open maps of locales $p:X\to B$ in terms of $\opens(B)$-modules (\cf\ section \ref{sec:localesheavesasmodules}): if $p$ is open, the homomorphism $\spp$ equals $u_!\circ p_!$. \qed
\end{proof}

If $Q$ is an inverse quantal frame and $X$ is an open $Q$-locale with $x\in \opens(X)$, we shall refer to $\spp(x)$ as the \emph{support} of $x$, and $\spp$ itself will be said to be the \emph{support} of $\opens(X)$, thus extending the terminology of section \ref{sec:localesheavesasmodules}. 

\begin{example}\label{openQloc}
Let $Q$ be an inverse quantal frame with base locale $B$. If $X$ is an open $B$-locale then $Q\otimes_B \opens(X)$ defines an open $Q$-locale (\cf\ \ref{QtensorX}). Its support is defined by $\spp(a\otimes x)=\spp(a\spp(x))$.
\end{example}

The following are useful properties of open $Q$-locales:

\begin{theorem}\label{lem:openGbundle}\label{lem:QLoctoBLoc}
Let $Q$ be an inverse quantal frame and let $X$ be an open $Q$-locale.
\begin{enumerate}
\item\label{bspp2} $\spp(ax) = \spp(a\spp(x))$ for all $a\in Q$ and $x\in \opens(X)$.
\item\label{bspp3} $\spp(ax) \le \spp(a)$ for all $a\in Q$ and $x\in \opens(X)$.
\item\label{conjugation} $\spp(sx)=s\spp(x)s^*$ for all $s\in \ipi(Q)$ and $x\in \opens(X)$.
\end{enumerate}
\end{theorem}

\begin{proof}
Denoting by $p$ and $\mathfrak a$ the projection and the action of the corresponding $G$-locale and using the equality $p\circ\mathfrak a = d\circ m\circ (1\times(u\circ p))$ of \ref{crucialemma} we prove \ref{bspp2}:
\[\spp(ax) = (u\circ p\circ\mathfrak a)_!(a\otimes x)
=(u\circ d\circ m\circ (1\times(u\circ p)))_!(a\otimes x) = \spp(a\spp(x))\;.\]
Then \ref{bspp3} follows immediately: $\spp(ax)=\spp(a\spp(x))\le\spp(ae)=\spp(a)$;
and \ref{conjugation} is a consequence of the inequalities $s\spp(x)s^*\le ss^*\le e$ and
\begin{eqnarray*}
\spp(sx)&=&\spp(s\spp(x))\le (s\spp(x))(s\spp(x))^*=s\spp(x)s^*\\
&=&\spp(s\spp(x)s^*)\le\spp(s\spp(x))\\
&=&\spp(sx)\;. \qed
\end{eqnarray*}

\end{proof}

\paragraph{Actions on local homeomorphisms.}

Let $G$ be an \'etale groupoid. A \emph{$G$-sheaf} is a $G$-locale whose projection is a local homeomorphism. The full subcategory of $G$-$\Loc$ whose objects are the $G$-sheaves (the classifying topos of $G$) is usually denoted by $BG$ and the isomorphism $G$-$\Loc\cong \opens(G)$-$\Loc$ yields, by restriction, a corresponding full subcategory:

\begin{definition}\label{def:Qsheaf}
Let $Q$ be an inverse quantal frame with base locale $B$. By an \emph{\'etale $Q$-locale}, or simply \emph{$Q$-sheaf}, is meant a (necessarily open) $Q$-locale $X$ such that the induced action of $B$ on $\opens(X)$ defines a $B$-sheaf. The \emph{category of \'etale $Q$-locales}, denoted by $Q$-$\Etale$, is the full subcategory of $Q$-$\Loc$ whose objects are the \'etale $Q$-locales.
\end{definition}

This is the natural definition, for an inverse quantal frame $Q$, of a ``$Q$-equivariant'' sheaf on $B$. We borrow the following terminology from the $B$-sheaves of \cite{RR}:

\begin{definition}
Let $Q$ be an inverse quantal frame with base locale $B$, and let $X$ be a $Q$-sheaf. The \emph{local sections} of $X$ are the local sections of $X$ regarded as a $B$-sheaf; that is, a local section is an element $s\in\opens(X)$ satisfying the equivalent conditions (\ref{eq:locsec1}) and (\ref{eq:locsec2}). The set of local sections of $X$ is denoted by $\sections_X$.
\end{definition}

\begin{example}
Any inverse quantal frame $Q$ itself defines a $Q$-sheaf and we have $\ipi(Q)\subset\sections_Q$.
\end{example}

\begin{example}
If $Q$ is an inverse quantal frame and $X$ is a $Q$-sheaf then $Q\otimes_B \opens(X)$ defines a $Q$-sheaf (\cf\ \ref{openQloc}).
\end{example}

An alternative notion of morphism of $Q$-sheaves, which maps local sections to local sections in the same way that a natural transformation between sheaves does, is the following (\cf\ the sheaf homomorphisms of section \ref{sec:localesheavesasmodules}):

\begin{definition}\label{def:shhom}
Let $Q$ be an inverse quantal frame and let $X$ and $Y$ be $Q$-sheaves. A \emph{sheaf homomorphism} $h:\opens(X)\to \opens(Y)$ is a homomorphism of left $Q$-modules that preserves supports and local sections; that is,
\begin{eqarray}
\spp(h(x))&=&\spp(x)\textrm{ for all }x\in \opens(X)\\
h(\sections_X)&\subset&\sections_Y\;.
\end{eqarray}%
The category of $Q$-sheaves and sheaf homomorphisms between them is denoted by $Q$-$\sh$.
\end{definition}

\begin{theorem}\label{lem:fullfaithfulsh}
Let $G$ be a localic \'etale groupoid and let $Q=\opens(G)$. The categories $BG$, $Q$-$\Etale$ and $Q$-$\sh$ are isomorphic.
\end{theorem}

\begin{proof}
$BG$ and $Q$-$\Etale$ are isomorphic by definition, so let us see that $Q$-$\Etale$ and $Q$-$\sh$ are isomorphic.
Let $X$ and $Y$ be $G$-sheaves with actions $\mathfrak a$ and $\mathfrak b$, respectively. If $f:X\to Y$ is a map of $G$-sheaves then $f$ is a local homeomorphism and $f_!$ is necessarily a sheaf homomorphism of $B$-sheaves (\cf\ \ref{sec:localesheavesasmodules}). From the equivariance condition
\begin{eq}\label{eq:Gequiv}
f\circ\mathfrak a=\mathfrak b\circ(1\times f)
\end{eq}%
we obtain, passing to direct images, the condition
\begin{eq}\label{eq:Qequiv}
f_!\circ\mathfrak a_!=\mathfrak 
b_!\circ (1\otimes f_!)
\end{eq}%
and thus $f_!$ is also a homomorphism of $Q$-modules. 
Therefore the assignment $f\mapsto f_!$ defines a faithful functor $F:Q$-$\Etale\to Q$-$\sh$ which is the identity on objects.

Now let
$h:\opens(X)\to \opens(Y)$ be an arbitrary sheaf homomorphism of $Q$-sheaves. This is also a sheaf homomorphism of $B$-sheaves and thus it is the direct image $f_!$ of a locale map $f:X\to Y$. The $Q$-equivariance of $h$ is therefore the condition (\ref{eq:Qequiv}). We obtain the inverse image homomorphism version of (\ref{eq:Gequiv}) by taking right adjoints, and thus we conclude that $F$ is full.
\qed
\end{proof}

\section{Groupoid sheaves as Hilbert modules}\label{sec:qsetshilbmod}\label{sec:gshm}

In this section we begin by studying the notion of complete Hilbert module, by which is meant a Hilbert quantale module equipped with a ``basis'', following which we establish an equivalence of quantaloids, for a given involutive quantale $Q$, between the quantaloid of $Q$-valued sets and the quantaloid of complete Hilbert $Q$-modules (\cf\ \ref{equivrelQQhmb}). Then we specialize the theory of complete Hilbert modules to supported quantales and finally we prove, for an \'etale groupoid $G$, that the $\opens(G)$-sheaves can be identified with the complete Hilbert $\opens(G)$-modules, which leads to the envisaged equivalence between the classifying topos $BG$ and the category of $\opens(G)$-valued sets.

\subsection{Hilbert bases}\label{sec:hilbertmodules}

The terminology ``Hilbert module'' was introduced by Paseka \cite{Paseka} as an adaptation to the context of involutive quantales of the notion of Hilbert C*-module (see \cite{Lance}), partly with the goal of relating aspects of the theory of operator algebras to quantales (see, \eg, \cite{Paseka2}). The notion of Hilbert basis appeared subsequently as a means of describing sheaves on involutive quantales and locales \cite{CT07,RR}, moreover in a way that meanwhile \cite{HS2} has been related in a precise way to (ordered) sheaves on non-involutive quantales via the notion of principally generated module of \cite{HS1}.

\paragraph{Basic definitions and examples.}

Let us begin by recalling the notion of Hilbert module:

\begin{definition}(\cite{Paseka})
Let $Q$ be an involutive quantale. By a \emph{pre-Hilbert $Q$-module} will be meant a left $Q$-module $X$ equipped with a binary operation
\[\inner --:X\times X\to Q\;,\]
called the \emph{inner product}, which for all $x, x_i,y\in X$ and $a\in Q$ satisfies the following axioms:
\begin{eqarray}
\inner{ax}y&=&a\inner x y\\
\left\langle\V_i x_i,y\right\rangle &=&\V_i\inner{x_i}y\\
\inner x y&=&\inner y x^*\;.
\end{eqarray}%
By a \emph{Hilbert $Q$-module} will be meant a pre-Hilbert $Q$-module whose inner product is ``non-degenerate'':
\begin{eq}
\inner x -=\inner y -\Rightarrow x=y\;.
\end{eq}%
\end{definition}

We remark that, in particular, inner products are ``sesquilinear forms'':
\begin{eq}
\left\langle x,\V a_i y_i\right\rangle=\V\inner x{y_i}a_i^*\;.
\end{eq}%

\begin{example}\label{exm:functionmodule}
$Q$ itself is a pre-Hilbert $Q$-module \cite{Paseka} with the inner product defined by
\[\inner a b=ab^*\;.\]
The inner product is non-degenerate if $Q$ is unital.
More generally, if $I$ is a set then the set $Q^I$ of maps $\vect v:I\to Q$ is a left $Q$-module with the usual function module structure given by pointwise joins and multiplication on the left, and it is a pre-Hilbert module with the inner product $\inner {\vect v} {\vect w}=\vect v\cdot \vect w$ given by the standard dot product formula
\[\vect v\cdot \vect w=\V_{\alpha\in I} v_{\alpha} w_{\alpha }^*\;.\]
(We adopt, for functions in $Q^I$ and their values, the same notation as for vectors and their components in linear algebra --- \cf\ section \ref{sec:mvhb}.)
\end{example}

\begin{example}\label{exm:hilbmodlocsecs}
If $p:X\to B$ is a local homeomorphism of locales then $\opens(X)$ is a Hilbert $\opens(B)$-module whose inner product is defined by $\inner x y=p_!(x\land y)$  --- \cf\ (\ref{firstip}).
\end{example}

\paragraph{Hilbert sections.}

Let us pursue the analogy between Hilbert modules and sheaves suggested by \ref{exm:hilbmodlocsecs} and define what should be meant in general by a ``section'' of a Hilbert module, using \ref{lem:firsthilbertsections} as motivation:

\begin{definition}
Let $Q$ be an involutive quantale and let $X$ be a pre-Hilbert $Q$-module. By a \emph{Hilbert section} of $X$ is meant an element $s\in X$ such that $\inner x s s\le x$ for all $x\in X$. The set of all the Hilbert sections of $X$ is denoted by $\sections_X$. We say that the Hilbert module $X$ is \emph{complete}, or that it \emph{has enough sections}, if for all $x\in X$ we have the equality
\[x=\V_{s\in\sections_X} \inner x s s\;.\]
Any set $\sections\subset X$ such that
$x=\V_{s\in\sections}\inner x s s$ for all $x\in X$
is called a \emph{Hilbert basis} (in particular, we have $\sections\subset\sections_X$ and $\sections$ is a set of $Q$-module generators for $X$).
\end{definition}

The name ``Hilbert basis'' is suggested by the obvious formal resemblance with the properties of a Hilbert basis of a Hilbert space. Of course, a Hilbert basis in our sense is not an actual basis as in linear algebra because there is no freeness, but for the sake of simplicity and following \cite{RR} we retain this terminology.

\begin{example}\label{exm:hilbsec1}
The Hilbert $\opens(B)$-module determined by a local homeomorphism of locales $p:X\to B$ (\cf\ \ref{exm:hilbmodlocsecs}) is complete with $\sections_X$ as a Hilbert basis \cite{RR}. Furthermore, $\sections_X$ is an actual basis of $X$ in the sense of locale theory (the analogue for locales of a basis of a topological space).
\end{example}

\begin{example}\label{exm:hilbsec2}
If $Q$ is a unital involutive quantale, then $Q$ itself, regarded as a Hilbert $Q$-module with $\inner a b=ab^*$ as in \ref{exm:functionmodule},
has a set of Hilbert sections \[\sections_Q=\{s\in Q\st s^*s\le e\}\;.\]
This set is a Hilbert basis, and so is
the singleton $\sections=\{e\}$.
\end{example}

\begin{example}\label{exm:hilbsec3}
The condition $\V\sections_X=1$ of \ref{exm:hilbsec1} does not necessarily hold over more general quantales. In order to see this let $R$ be the unital involutive quantale whose involution is trivial and whose order and multiplication table are the following (\cf\ \cite[Example 4.21]{Re07}):
\[
\vcenter{\xymatrix{
&1\ar@{-}[dl]\ar@{-}[dr]\\
e\ar@{-}[dr]&&a\ar@{-}[dl]\\
&0
}}\hspace*{2cm}
\begin{array}{c|cccc}
	&  0 & e & a & 1 \\
	\hline
	0 &  0 & 0 & 0 & 0 \\
	e &  0 & e & a & 1 \\
	a &  0 & a & 1 & 1 \\
	1 &  0 & 1 & 1 & 1
	\end{array}
\]
If we regard $R$ as a Hilbert $R$-module with $\inner x y=xy^*$ then $R$ has enough sections (because it is a unital quantale) but $\sections_R=\{0,e\}$.
\end{example}

The existence of a Hilbert basis has useful consequences. In particular the inner product is necessarily non-degenerate:

\begin{lemma}\label{usefullemma}
Let $Q$ be an involutive quantale, let $X$ be a pre-Hilbert $Q$-module, and let $\sections\subset X$. If $\sections$ is a Hilbert basis then the following properties hold, for all $x,y\in X$.
\begin{enumerate}
\item\label{ul2} If $\inner x s=\inner y s$ for all $s\in \sections$ then $x=y$. (Hence, $X$ is a Hilbert module.)
\item\label{ul3} $\inner x y=\V_{s\in\sections}\inner x s\inner s y=\V_{s\in\sections}\inner x s\inner y s^*$. (``Parseval's identity''.)
\end{enumerate}
Conversely, $\sections$ is a Hilbert basis if $\inner--$ is non-degenerate and \ref{ul3} holds.
\end{lemma}

\begin{proof}
Assume that $\sections$ is a Hilbert basis. The two properties are proved as follows.
\begin{enumerate}
\item[{\rm \ref{ul2}.}] If $\inner x s=\inner y s$ for all $s\in \sections$ then
$x=\V_{s\in\sections}\inner x s s=\V_{s\in\sections}\inner y s s=y$.
\item[{\rm \ref{ul3}.}]
$\inner x y=\left\langle{\V_{s\in\sections}\inner x s s},y\right\rangle=\V_{s\in\sections}\inner x s\inner s y$.
\end{enumerate}
For the converse assume that $\inner - -$ is non-degenerate and that \ref{ul3} holds. Then for all $x,y\in X$ we have
\[\left\langle{\V_{s\in\sections}\inner x s s},y\right\rangle=\V_{s\in\sections}\inner x s\inner s y=\inner x y\;,\]
and by the non-degeneracy we obtain $\V_{s\in\sections}\inner x s s=x$. \qed
\end{proof}

\paragraph{Adjointable maps.}

Similarly to Hilbert C*-modules, the module homomorphisms which have ``operator adjoints'' play a special role:

\begin{definition}\textbf{(\cite{Paseka})}
Let $Q$ be an involutive quantale and let $X$ and $Y$ be pre-Hilbert $Q$-modules. A function \[\varphi:X\to Y\] is \emph{adjointable} if there is another function $\varphi^\dagger:Y\to X$ such that for all $x\in X$ and $y\in Y$ we have
\[\inner{\varphi(x)}y = \inner x{\varphi^\dagger(y)}\;.\]
(The notation for $\varphi^\dagger$ in \cite{Paseka} is $\varphi^*$, but we want to avoid confusion with the notation for inverse image homomorphisms of locale maps.)
\end{definition}

We note that if $\varphi$ is adjointable and $Y$ is a Hilbert $Q$-module (\ie, the bilinear form of $Y$ is non-degenerate) then $\varphi$ is necessarily a homomorphism of $Q$-modules \cite{Paseka}: we have
\begin{eqnarray*}
\left\langle \varphi\left(\V a_i x_i\right),y\right\rangle
&=&\left\langle \V a_i x_i,\varphi^\dagger(y)\right\rangle
=\V a_i\langle x_i, \varphi^\dagger(y)\rangle\\
&=&\V a_i\langle \varphi(x_i),y\rangle
=\left\langle\V a_i \varphi(x_i),y\right\rangle
\end{eqnarray*}
and thus by the non-degeneracy of $\inner--_Y$ we conclude that
\[\varphi\left(\V a_i x_i\right)=\V a_i \varphi(x_i)\;.\]

Conversely, and similarly to the situation in \cite{RR} where $Q$ was a locale, the homomorphisms of complete Hilbert $Q$-modules are necessarily adjointable. In order to prove this only the domain module need have enough sections:

\begin{theorem}
Let $Q$ be an involutive quantale and let $X$ and $Y$ be pre-Hilbert $Q$-modules such that $X$ has a Hilbert basis $\sections$ (hence, $X$ is a Hilbert module), and let $\varphi:X\to Y$ be a homomorphism of $Q$-modules. Then $\varphi$ is adjointable with a unique adjoint $\varphi^\dagger$, which is given by
\begin{eq}\label{adjoint}
\varphi^\dagger(y) = \V_{t\in\sections}\langle y,\varphi(t)\rangle t\;.
\end{eq}%
\end{theorem}

\begin{proof}
Let $x\in X$, $y\in Y$, and let us compute $\langle x,\varphi^\dagger(y)\rangle$ using (\ref{adjoint}):
\begin{eqnarray*}
\langle x,\varphi^\dagger(y)\rangle &=&\left\langle \V_{s\in\sections}\langle x,s\rangle s,\V_{t\in\sections}\langle y,\varphi(t)\rangle t\right\rangle\\
&=&
\V_{s,t\in\sections}\inner x s\inner s t\inner y{\varphi(t)}^*\\
&=&
\V_{t\in\sections}\inner x t\inner {\varphi(t)}y\\
&=&
\left\langle \V_{t\in\sections}\inner x t \varphi(t),y\right\rangle\\
&=&
\left\langle \varphi\left(\V_{t\in\sections}\inner x t t\right),y\right\rangle\\
&=&
\inner {\varphi(x)}y\;.
\end{eqnarray*}
This shows that $\varphi^\dagger$ is adjoint to $\varphi$, and the uniqueness is a consequence of the non-degeneracy of the inner product of $X$. \qed
\end{proof}

\begin{definition}
Let $Q$ be an involutive quantale. The \emph{category of complete Hilbert $Q$-modules}, denoted by $Q$-$\hmb$ (standing for `Hilbert Modules with Basis'), is the category whose objects are the complete Hilbert $Q$-modules and whose arrows are the homomorphisms of $Q$-modules (equivalently, the adjointable maps).
\end{definition}

\begin{corollary}
For any involutive quantale $Q$, $Q$-$\hmb$ is an involutive quantaloid whose involution is the strong self-duality $(-)^\dagger:\opp{(Q\textrm{-}\hmb)}\to Q\textrm{-}\hmb$.
\end{corollary}

\begin{example}\label{exm:BsheavesasHilbertmodules}
If $B$ is a frame, the category $B$-$\sh$ (\cf\ section \ref{sec:localesheavesasmodules}) coincides with $\maps(B$-$\hmb)$. Moreover, if $f:X\to Y$ is a map of $B$-sheaves we have $f_!=(f^{*})^{\dagger}$ \cite[Theorem 8]{RR}.
\end{example}

\subsection{Quantale-valued sets}\label{sec:mvhb}

Most of the definitions of sheaf for involutive quantales in the literature are based on generalizations of the notion of frame-valued set. In this section we take one such definition and show, for an arbitrary involutive quantale $Q$, that the category of $Q$-valued sets is equivalent to the category of maps of the quantaloid of complete Hilbert $Q$-modules.

\paragraph{Basic definitions.}

Let $Q$ be an involutive quantale, and let $I$ and $J$ be sets. By a \emph{$Q$-valued matrix of type $I\times J$} will be meant a mapping
\[A:I\times J\to Q\;.\]
Terminology and notation are analogous to those of linear algebra:
\begin{itemize}
\item $a_{\alpha\beta}$ or $(A)_{\alpha\beta}$ denotes the value $A(\alpha,\beta)$, and we refer to $\alpha$ and $\beta$ as the \emph{row} and \emph{column} indices, respectively;
\item if $I$ is a singleton we say that $A$ is a \emph{row matrix}, and if $J$ is a singleton we say that $A$ is a \emph{column matrix};
\item the \emph{transpose} of $A$ is the matrix $A^T$ defined by $(A^T)_{\alpha\beta}=a_{\beta\alpha}$, and the \emph{adjoint} of $A$ is the matrix $A^*$ defined by $(A^*)_{\alpha\beta}=a^*_{\beta\alpha}$;
\item if $B:J\times K\to Q$ is another matrix, the \emph{product} $AB$ is defined by \[(AB)_{\alpha\gamma} = \V_{\beta\in J}a_{\alpha\beta}b_{\beta\gamma}\;.\]
\end{itemize}

In accordance with these conventions, we shall often think of a mapping $\vect v: I\to Q$ as being a ``vector'' (\cf\ \ref{exm:functionmodule}), in particular writing $v_\alpha$ instead of $\vect v(\alpha)$ and adopting the following notation and terminology with respect to a matrix $A:I\times J\to Q$:
\begin{itemize}
\item $\vect v A:J\to Q$ is the mapping whose components are defined by \[(\vect v A)_{\beta}=\V_{\alpha\in I} v_{\alpha}a_{\alpha\beta}\;;\]
that is, $\vect v$ is always regarded as being a row matrix $\{*\}\times I\to Q$;
\item if $\alpha\in I$, the \emph{$\alpha$-row} of $A$ is the mapping $\tilde\alpha:I\to Q$ defined by $\tilde\alpha_{\beta}=a_{\alpha\beta}$;
\item if $\beta\in J$, the \emph{$\beta$-column} of $A$ is the mapping $\tilde\beta:J\to Q$ defined by $\tilde\beta_{\alpha}=a_{\alpha\beta}$.
\end{itemize}

The following definition stems from some of the early works on categories of relations \citelist{\cite{Allegories}*{\S 2.226} \cite{Pitts}*{Prop.\ 2.6}}. We adapt it from Garraway \cite{Garraway}, who applies it to non-unital involutive quantaloids rather than just involutive quantales. (The same definition has been used earlier by Gylys \cite{G01} for unital involutive quantaloids.)

\begin{definition}\label{Qsets}
Let $Q$ be an involutive quantale. By a \emph{$Q$-set} is meant a set $I$ together with a matrix
\[A:I\times I\to Q\]
which is both \emph{self-adjoint} ($A=A^*$) and \emph{idempotent} ($A A=A$).
\end{definition}

The matrix entry $a_{\alpha\beta}$ of a $Q$-set $(I,A)$ can be regarded as the generalized truth-value of the equality $\alpha=\beta$, and the diagonal entry $a_{\alpha \alpha}$ is regarded as the extent to which the element $\alpha $ exists. The fact that $A$ is required to be a projection matrix reflects the fact that equality should be a partial equivalence relation.
These ideas lead naturally to the following definition of a $Q$-valued relation between $Q$-sets (another common name for this is \emph{distributor} or \emph{bimodule}):

\begin{definition}
Let $Q$ be an involutive quantale and let $X=(I,A)$ and $Y=(J,B)$ be $Q$-sets.
A \emph{relation} \[R:X\relto Y\] is a matrix $R:J\times I\to Q$ such that the following equations hold:
\begin{eq}
BR = R = RA\;. \label{projectionsplitting}
\end{eq}%
We shall denote by $\Rel(Q)$ the quantaloid whose objects are the $Q$-sets and whose morphisms are the relations between them: composition is given by matrix multiplication and the identity relation on a $Q$-set $(I,A)$ is $A$.
\end{definition}

From this a notion of map immediately follows (\cf\ \ref{mapsininvqs}):

\begin{definition}
Let $Q$ be an involutive quantale. The \emph{category of $Q$-sets} $\sets(Q)$ is defined to be $\Map(\Rel(Q))$. Hence, explicitly, a \emph{map} of $Q$-sets $F:(I,A)\to (J,B)$ is a relation such that the following two additional conditions hold:
\begin{eqarray}
F F^*&\le& B \label{singlevalued}\\
A&\le& F^* F\;. \label{totalfunction}
\end{eqarray}%
(Note: we write $\relto$ for general relations and $\to$ for maps.)
\end{definition}

\paragraph{Matrices versus modules.}
Every Hilbert $Q$-module $X$ with a Hilbert basis $\sections$ has an associated matrix $A:\sections\times\sections\to Q$ defined by $a_{st}=\inner s t$ (this is analogous to the metric of an Euclidian space with respect to a chosen basis). This defines a $Q$-set $(\sections,A)$ because we have $A=A^*$ by definition of the inner product, and $A=A^2$ by ``Parseval's identity'' (\cf\ \ref{usefullemma}-\ref{ul3}). Conversely, we have:

\begin{lemma}\label{lemma:matrixtomod}
Let $Q$ be an involutive quantale. For any $Q$-set $(I,A)$ the subset of $Q^I$ defined by
\[Q^I A=\{\vect v A\st \vect v\in Q^I\}\]
is a Hilbert $Q$-module whose inner product is the dot product of $Q^I$ (\cf\ \ref{exm:functionmodule}),
\[\langle\vect v,\vect w\rangle = \vect v\cdot\vect w=\vect v\vect w^* = \V_{\alpha\in I} v_{\alpha} w_{\alpha }^*\;,\]
it has a Hilbert basis $\sections$ consisting of all the rows of $A$, and for all $\alpha,\beta\in I$ and $\vect v\in Q^I$ we have
\begin{eqarray}
\left\langle \vect v,\tilde \beta\right\rangle &=& (\vect v A)_{\beta}\;, \label{proj}\\
\left\langle{\tilde \alpha},{\tilde \beta}\right\rangle &=& a_{\alpha\beta}\;.\label{matrixentry}
\end{eqarray}%
\end{lemma}

\begin{proof}
The assignment $j:\vect v\mapsto \vect v A$ is a $Q$-module endomorphism of $Q^I$, and $Q^I A$ is its image, hence a submodule of $Q^I$. Next note that $\sections$ is a subset of $Q^I A$ because for each $\alpha\in I$ we have
$\tilde \alpha =\tilde \alpha A\in Q^I A$:
\[\tilde \alpha_\beta=a_{\alpha\beta}=(A^2)_{\alpha\beta}=\V_{\gamma\in I} a_{\alpha\gamma} a_{\gamma\beta} = \V_{\gamma\in I} \tilde \alpha_{\gamma} a_{\gamma\beta} =(\tilde \alpha A)_\beta \;.\]
Now we prove (\ref{proj}):
\[\left\langle \vect v,\tilde \beta\right\rangle=\vect v \cdot{\tilde \beta}=\V_{\gamma} v_{\gamma}\tilde \beta_{\gamma}^*=\V_{\gamma} v_{\gamma} a_{\beta\gamma}^*=\V_{\gamma} v_{\gamma} a_{\gamma\beta}=(\vect v A)_\beta\;.\]
In particular, if $\vect v\in Q^I A$ we have $\vect v A=\vect v$, hence $\left\langle \vect v,\tilde \beta\right\rangle=v_\beta$, and (\ref{matrixentry}) is an immediate consequence:
\[\left\langle \tilde \alpha,\tilde \beta\right\rangle = \tilde \alpha_\beta = a_{\alpha\beta}\;.\]
Finally, $\sections$ is a Hilbert basis because for all $\vect v\in Q^I A$ we have
\[\left(\V_{\beta} \left\langle \vect v,\tilde \beta\right\rangle\tilde \beta\right)_{\alpha} = \left(\V_{\beta} v_{\beta}\tilde \beta\right)_{\alpha}
=\V_{\beta} v_{\beta} \tilde \beta_{\alpha} = \V_{\beta} v_{\beta} a_{\beta \alpha} = (\vect v A)_{\alpha }=v_{\alpha}\;. \qed\]
\end{proof}

\begin{theorem}\label{equivrelQQhmb}
Let $Q$ be an involutive quantale. The involutive quantaloids $\Rel(Q)$ and $Q$-$\hmb$ are equivalent.
\end{theorem}

\begin{proof}
For each complete Hilbert $Q$-module $X$ let $G(X)=(\sections_X,A_X)$ be the $Q$-set defined by $(A_X)_{st}=\inner s t$, and for each homomorphism $\varphi:X\to Y$ in $Q$-$\hmb$ let $G(\varphi):\sections_Y\times\sections_X\to Q$ be the matrix defined by
$(G(\varphi))_{st} = \inner s{\varphi(t)}_Y$.
It is straightforward to see that these assignments define a faithful homomorphism of quantaloids
$G:Q\textrm{-}\hmb\to \Rel(Q)$
(in particular faithfulness is a consequence of the non-degeneracy of the inner products), and we prove only that $G$ preserves composition: for all homomorphisms $\varphi:X\to Y$ and $\psi:Y\to Z$ we have
\begin{eqnarray*}
(G(\psi\circ\varphi))_{st} &=& \inner s{\psi(\varphi(t))}_Z=\inner{\psi^\dagger(s)}{\varphi(t)}_Y\\
&=&\V_{u\in\sections_Y}\inner{\psi^\dagger(s)}u_Y\inner u{\varphi(t)}_Y=\V_u\inner s{\psi(u)}_Z\inner u{\varphi(t)}_Y\\
&=&\V_u (G(\psi))_{su}(G(\varphi))_{ut}=
(G(\psi)G(\varphi))_{st}\;.
\end{eqnarray*}
We remark that this also shows that $G(\varphi)$ is a morphism in $\Rel(Q)$:
\[A_Y G(\varphi)=G(\ident_Y)G(\varphi)=G(\ident_Y\circ\varphi)=G(\varphi)=G(\varphi)G(\ident_X)=G(\varphi)A_X\;.\]

Next we show, for an arbitrary $Q$-set $(I,A)$, that $(I,A)\cong G(Q^I A)$. Writing $\vect X$ for $Q^I A$, let $R:\sections_{\vect X}\times I\to Q$ be the matrix given by $r_{\vect\sigma \alpha} = \inner{\vect\sigma} {\tilde \alpha}=\sigma_{\alpha }$. This defines a relation $R:(I,A)\relto G(\vect X)$:
\[\begin{array}{rcl}
(A_{\vect X} R)_{\vect\sigma \alpha}&=& \V_{\vect\tau\in\sections_{\vect X}} \inner{\vect\sigma}{\vect\tau}\tau_{\alpha }=
\left(\V\nolimits_{\vect\tau} \inner{\vect\sigma}{\vect\tau}\vect\tau\right)_{\alpha }=\sigma_{\alpha }=r_{\vect\sigma \alpha}\;,\\
~\\
(R A)_{\vect\sigma \alpha} &=& \V_{\beta\in I} \inner{\vect\sigma} {\tilde \beta} a_{\beta \alpha}=
\V\nolimits_{\beta} \inner{\vect\sigma} {\tilde \beta} \tilde \beta_{\alpha }=\left(\V\nolimits_{\beta}
\inner{\vect\sigma} {\tilde \beta} \tilde \beta\right)_{\alpha }\\
&=&\sigma_{\alpha }=r_{\vect\sigma \alpha}\;.
\end{array}\]
The above two lines are justified, respectively, because $\sections_{\vect X}$ and $\{\tilde \alpha\st \alpha\in I\}$ are Hilbert bases of ${\vect X}$. The same facts show that $R$ is a unitary map:
\[\begin{array}{rcl}
(RR^*)_{\vect\sigma \vect\tau} &=&
\V_{\alpha\in I} \inner{\vect\sigma}{\tilde \alpha}\inner{\vect\tau}{\tilde \alpha}^*
=\V\nolimits_{\alpha} \inner{\vect\sigma}{\tilde \alpha}\inner{\tilde \alpha}{\vect\tau}=\inner{\vect\sigma}{\vect\tau}=(A_{\vect X})_{\vect\sigma\vect\tau}\;,\\
~\\
(R^*R)_{\alpha\beta}&=& \V\nolimits_{\vect\sigma\in\sections_X} \inner{\tilde \alpha}{\vect\sigma}\inner{\vect\sigma}{\tilde \beta}
=\inner{\tilde \alpha}{\tilde \beta}=a_{\alpha\beta}
\;.
\end{array}\]

Now let us prove that $(R,\vect X)$ is a universal arrow from $(I,A)$ to $G$. Let $Y$ be a complete Hilbert $Q$-module and let $H:(I,A)\relto G(Y)$. 
Define a mapping $\varphi:{\vect X}\to Y$ as follows:
\begin{eq}\label{adjunctformula}
\varphi(\vect v) = \V_{\scriptsize \begin{array}{l}s\in I\\ t\in\sections_Y\end{array}} v_s h_{ts}^* t\;.
\end{eq}%
This is a homomorphism of left $Q$-modules because $\vect v A=\vect v$ and, by (\ref{proj}), $v_s=\inner{\vect v}{\tilde s}$. In order to prove the universal property we only need to show (because $R$ is an isomorphism and $G$ is faithful) that $G(\varphi)R=H$: for all $\alpha\in\sections_Y$ and all $\beta\in I$ we have
\begin{eqnarray*}
(G(\varphi)R)_{\alpha\beta} &=& \V_{\vect\sigma\in\sections_{\vect X}}\inner\alpha{\varphi(\vect\sigma)}\inner{\vect\sigma}{\tilde\beta} = \V_{\vect\sigma\in\sections_{\vect X}}\inner{\varphi^\dagger(\alpha)}{\vect\sigma}\inner{\vect\sigma}{\tilde\beta}\\
&=& \inner{\varphi^\dagger(\alpha)}{\tilde\beta}=
\inner\alpha{\varphi(\tilde\beta)}\\
&=& \left\langle \alpha,\V_{\scriptsize \begin{array}{l}s\in I\\ t\in\sections_Y\end{array}} \tilde\beta_s h_{ts}^* t\right\rangle = \V_{\scriptsize \begin{array}{l}s\in I\\ t\in\sections_Y\end{array}}\langle \alpha,\tilde\beta_s h_{ts}^* t\rangle \\
&=&\V_{\scriptsize \begin{array}{l}s, t\end{array}}\langle \alpha,t\rangle h_{ts}\tilde\beta_s^*=
\V_{\scriptsize \begin{array}{l} s, t\end{array}}\langle \alpha,t\rangle h_{ts}\inner{\tilde s}{\tilde\beta}\\
&=&(A_Y H A)_{\alpha\beta}=h_{\alpha\beta}\;.
\end{eqnarray*}

We thus conclude that $G$ has a left adjoint and that the unit of the adjunction is a unitary map. Therefore, in order to establish the desired equivalence we only need to prove that the co-unit is an isomorphism (\cf\ comments after \ref{quantaloidequiv}). So let $X$ be an arbitrary complete Hilbert $Q$-module, and,  for simplicity, let us write $G(X)=(\sections,A)$ instead of $(\sections_X,A_X)$. Let $\varphi:Q^{\sections}\to X$ be the $Q$-module quotient defined by $\varphi(\vect v) = \V_{s\in \sections} v_s s$. By \ref{usefullemma} we have,  for all $\vect v,\vect w\in Q^\sections$, the following equivalences:
\begin{eqnarray*}
\varphi(\vect v)=\varphi(\vect w) &\iff& \forall_{t\in\sections}\ \inner{\varphi(\vect v)}t=\inner{\varphi(\vect w)}t\\
&\iff&\forall_{t\in\sections}\ \left\langle\V\nolimits_{s\in \sections} v_s s,t\right\rangle = \left\langle\V\nolimits_{s} w_s s,t\right\rangle\\
&\iff& \forall_{t\in\sections}\ \V\nolimits_{s} v_s\inner s t=\V\nolimits_{s} w_s\inner s t\\
&\iff& \forall_{t\in\sections}\  (\vect v A)_t=(\vect w A)_t\\
&\iff& \vect v A=\vect w A\;.
\end{eqnarray*}
Hence, $\varphi$ factors uniquely through the quotient $\vect v\mapsto \vect v A : Q^{\sections}\to Q^\sections A$ and an isomorphism of $Q$-modules $Q^\sections A\stackrel\cong\to X$, which is the $X$-component of the co-unit of the adjunction. \qed
\end{proof}

\begin{corollary}
For any involutive quantale $Q$, the category $\sets(Q)$ is equivalent to $\maps(Q$-$\hmb)$  (\cf\ \ref{mapsininvqs}).
\end{corollary}

\subsection{Supported modules}\label{sec:ssq}

Let us provide an independent study of Hilbert modules on supported quantales. This has two purposes: one is to achieve a better understanding of how the various axioms interact with each other; and the other is that by doing so one is paving the way for obtaining possible extensions of the theory developed in this paper in a way that may be applicable to theories of sheaves on supported quantales that are more general than inverse quantal frames. As an example, we mention the Lindenbaum quantales for propositional normal modal logic of \cite{MarcR}, which are stably supported and whose sheaves may provide good semantic grounds for interpreting non-propositional modal logic.

\paragraph{Modules on supported quantales.}

We begin with a simple but useful property of arbitrary modules on supported quantales:

\begin{lemma}\label{lem:modprop}
Let $Q$ be a supported quantale and let $X$ be a left $Q$-module. Then for all $a\in Q$ we have
\[a1_X=\spp(a)1_X=aa^*1_X=aa^*a1_X=aa^*aa^*1_X\;.\]
\end{lemma}

\begin{proof}
Let $a\in Q$. The axioms of supported quantales give us
\begin{eqnarray*}
a1_X&\le&\spp(a)a1_X\le\spp(a)1_X\le aa^*1_X\le\spp(a)aa^*1_X\le aa^*aa^*1_X\le aa^*a1_X\\
&\le& a1_X\;. \qed
\end{eqnarray*}
\end{proof}

Let us introduce a notion of support for modules that is formally similar to that of quantales if we replace $ab^*$ by $\inner a b$ (however, we require supports to be only monotone instead of sup-preserving --- \cf\ \ref{rem:weakspp}):

\begin{definition}
Let $Q$ be a supported quantale with base locale $B$ (\cf\ \ref{def:baselocale}). By a \emph{supported $Q$-module} is meant a pre-Hilbert $Q$-module $X$ equipped with a monotone map
\[\spp: X\to B\;,\]
called the \emph{support} of $X$, such that the following properties hold for all $x\in X$:
\begin{eqnarray*}
\spp (x) &\le& \langle x,x\rangle\\
x &\le& \spp (x) x\;.
\end{eqnarray*}
\end{definition}

\begin{example}
Any supported quantale $Q$ defines a supported module over itself, with $\langle a,b\rangle=ab^*$.
\end{example}

The existence of the base locale $B$ enables us to define a notion of local section in analogy to that of local homeomorphisms regarded as $B$-modules (\cf\ sections \ref{sec:localesheavesasmodules} and \ref{sec:actionsonsheaves}), since the action of $Q$ restricts to an action of $B$ making $X$ a supported $B$-module:

\begin{definition}\label{exm:localsection}
Let $Q$ be a supported quantale and let $X$ be a supported $Q$-module. By a \emph{local section} of $X$ is meant an element $s\in X$ such that $\spp(x)s=x$ for all $x\le s$. The set of local sections of $X$ is denoted by $\sections_X^\ell$.
\end{definition}
But, as we shall see, in the cases of interest later in the paper the local sections coincide with the Hilbert sections due to the following proposition:

\begin{lemma}\label{lem:localsectionsprops}
Let $Q$ be a supported quantale and let $X$ be a supported $Q$-module.
\begin{enumerate}
\item\label{equivdefls} $\sections_X^\ell=\{s\in X\st \spp(x\land s)s\le x\textrm{ for all }x\in X\}$.
\item\label{numbertwo} $\sections_X^\ell$ is downwards closed.
\item\label{numberthree} $\sections_X\subset \sections_X^\ell$.
\item\label{hilbeqloc} $\sections_X=\sections_X^\ell$ if and only if every local section $s$ is a join $s=\V_i t_i$ of Hilbert sections.
\end{enumerate}
\end{lemma}

\begin{proof}
\ref{equivdefls} is proved in the same way as the equivalence of (\ref{eq:locsec1}) and (\ref{eq:locsec2}) in \ref{lem:firsthilbertsections}: if $s$ is a local section and $x\in X$ then $x\land s\le s$ and thus we have
$\spp(x\land s)s=x\land s\le x$; conversely, if $s$ satisfies $\spp(x\land s)s\le x$ for all $x\in X$ then if $x\le s$ we have
$x=\spp(x) x\le\spp(x)s=\spp(x\land x) s\le\spp(x\land s) s\le x$. Now \ref{numbertwo} is an immediate consequence because if $s$ is a local section and $t\le s$ we have $\spp(x\land t)t\le\spp(x\land s)s\le x$. Similarly, \ref{numberthree} follows from the inequality $\spp(x\land s)\le\inner x s$.
In order to prove the nontrivial inclusion in \ref{hilbeqloc} let $s$ be a local section and let $I\subset\sections_X$ be such that $s=\V I$. Let $t$ and $u$ be arbitrary elements of $I$. For all $x\in X$ we have $\spp(x\land s) s\le x$, and thus also
$\spp( x\land t) u\le x$. In particular, making $x=t$ we obtain $\spp(t)u\le t$. The conclusion that $s$ is a Hilbert section follows immediately, since for all $x\in X$ we have
$\inner x s s=\V_{t,u\in I}\inner x t u$ and
$\inner x t u=\inner x{\spp(t)t}u=\inner x t\spp(t)u\le\inner x t t\le x$. \qed
\end{proof}

\paragraph{Stably supported modules.}
The notion of stable support for quantales has an equally useful analogue for modules:

\begin{definition}
Let $Q$ be a supported quantale and let $X$ be a supported $Q$-module. The support is called \emph{stable}, and the module is said to be \emph{stably supported}, if in addition the support is $B$-equivariant; that is, the following condition holds for all $b\in B$ and $x\in X$:
\[\spp(bx) = b\land\spp(x)\;.\]
\end{definition}

\begin{example}
Any stably supported quantale $Q$ is itself a stably supported $Q$-module, with $\langle a,b\rangle=ab^*$.
\end{example}

\begin{lemma}
Let $Q$ be a supported quantale and let $X$ be a stably supported $Q$-module. The map $b\mapsto b1_X$ from $B$ to $X$ is right adjoint to $\spp:X\to B$. (In particular, $\spp$ preserves joins.)
\end{lemma}

\begin{proof}
The unit of the adjunction follows from $x\le \spp(x)x\le\spp(x)1_X$, and the co-unit follows from the $B$-equivariance:
$\spp(b1_X)=b\land\spp(1_X)\le b$. \qed
\end{proof}

There are several alternative definitions of stability:

\begin{lemma}\label{lem:adjoint}
Let $Q$ be a supported quantale and let $X$ be a supported $Q$-module. The following conditions are equivalent:
\begin{enumerate}
\item\label{lem:stabclassical1} $\spp(ax) = \spp(a\spp( x))$ for all $a\in Q$ and $x\in X$;
\item\label{lem:stabclassical2} $\spp(ax)\le\spp (a)$ for all $a\in Q$ and $x\in X$;
\item\label{lem:stabclassical3} $\spp(a1_X)\le\spp (a)$ for all $a\in Q$;
\item\label{lem:Bequiv} The support of $X$ is stable.
\end{enumerate}
\end{lemma}

\begin{proof}
\ref{lem:stabclassical2} and \ref{lem:stabclassical3} are of course equivalent. Let us prove the equivalence of \ref{lem:stabclassical1} and \ref{lem:stabclassical2}. First assume \ref{lem:stabclassical1}. Then for all $a\in Q$ and $x\in X$ we have
\[\spp(ax) = \spp(a\spp (x)) \le \spp(ae) = \spp (a)\;.\]
Now assume \ref{lem:stabclassical2}. Then we have
\begin{eqnarray*}
\spp(ax)&\le&\spp(a\spp (x) x)\le\spp(a\spp (x))\le\spp(a\langle x,x\rangle)
=\spp(\langle ax,x\rangle)\\
&\le&\spp(\langle\spp(ax)ax,x\rangle)
=\spp(\spp(ax)\langle ax,x\rangle)\le\spp(\spp(ax))\\
&=&\spp(ax)\;,
\end{eqnarray*}
and thus \ref{lem:stabclassical1} holds. Now assume again \ref{lem:stabclassical1}, and let $b\in B$. Then
\[\spp(bx)=\spp(b\spp(x))=\spp(b\land\spp(x))=b\land\spp(x)\;,\]
and thus we see that \ref{lem:stabclassical1} implies \ref{lem:Bequiv}. Finally, assume that \ref{lem:Bequiv} holds. Then for all $a\in Q$ we have, using \ref{lem:modprop},
\[\spp(a1_X)=\spp(\spp(a)1_X)=\spp(a)\land\spp(1_X)\le\spp(a)\;,\]
and thus \ref{lem:Bequiv} implies \ref{lem:stabclassical3}. \qed
\end{proof}

\begin{remark}\label{previouslyunnoticed}
Since a stably supported quantale is a stably supported module over itself and \ref{lem:adjoint}-\ref{lem:stabclassical1} translates to the definition of stability for supported quantales [\cf\ (\ref{spp4}], it follows that a supported quantale $Q$ is stably supported if and only if its support is $B$-equivariant (\cf\ \ref{thm:stabiffequiv}). This fact is not mentioned in \cite{Re07}.
\end{remark}

\paragraph{Supported modules on stably supported quantales.}

In order for some of the good properties of stable supports of quantales to carry over to supports of modules we need the quantale $Q$ itself to be stably supported. First we observe the following:

\begin{lemma}\label{lem:necstab}
Let $Q$ be a stably supported quantale. Any supported $Q$-module is necessarily stably supported.
\end{lemma}

\begin{proof}
For all $a\in Q$ and $x\in X$ we necessarily have $\spp(ax)\le\spp (a)$, as the following sequence of (in)equalities shows,
\[\spp(ax)\le\langle ax,ax\rangle\wedge e=a\langle x,ax\rangle\wedge e\le a1_Q\wedge e=\spp (a)\;,\]
where the latter equality is (\ref{spp6}). \qed
\end{proof}

We obtain similar properties to those of stable supports of quantales regarding uniqueness of supports, and analogous formulas for $\spp$ when we formally replace $ab^*$ by $\langle a,b\rangle$:

\begin{lemma}
Let $Q$ be a stably supported quantale and let $X$ be a (necessarily stably) supported $Q$-module.
The following properties hold:
\begin{enumerate}
\item\label{lem:stabprop1} For all $x,y\in X$ we have $\spp(\langle x,y\rangle)\le\spp (x)=\spp(\langle x,x\rangle)=\spp(\langle x,1\rangle)$.
\item\label{lem:stabprop5} For all $x\in X$ and $a\in Q$ we have $\spp (x)a=\langle x,1\rangle\wedge a$.
\item\label{lem:stabprop6} For all $x\in X$ we have $\spp (x)=\langle x,1\rangle\wedge e$.
\item\label{lem:stabprop4} For all $x\in X$ we have $\spp (x)=\langle x,x\rangle\wedge e$.
\item $X$ does not admit any other support.
\end{enumerate}
\end{lemma}

\begin{proof}
First we prove \ref{lem:stabprop1}. Let $x,y\in X$. Using the stability of the support of $Q$ we have
\[\spp (\langle x,y\rangle)\le\spp(\langle\spp (x) x,y\rangle)
=\spp (\spp (x)\langle x,y\rangle)\le\spp(\spp (x)) =\spp (x)\;.\]
On the other hand, using the inequality just proved we obtain
\[\spp (x) = \spp(\spp (x)) \le \spp(\langle x,x\rangle)\le\spp(\langle x,1\rangle)\le\spp (x)\;,\]
thus proving  \ref{lem:stabprop1}.
Now let us prove \ref{lem:stabprop5}. We have $\spp(x)a\le\inner x 1$ because
\[\spp(x)a\le\inner x x 1\le\inner x 1 1=\inner x{1^*1}=\inner x 1\;.\]
Since we also have $\spp(x)a\le a$ we obtain the inequality
\[\spp(x)a\le \inner x 1\land a\;,\]
and the converse inequality is proved as follows, using \ref{lem:stabprop1}:
\[\inner x 1\land a\le\spp(\inner x 1\land a)(\inner x 1\land a)\le \spp(\inner x 1) a
=\spp(x)a\;.\]
Making $a=e$ we obtain \ref{lem:stabprop6} (which immediately implies that the support of $X$ is unique), and for \ref{lem:stabprop4} it suffices to prove the inequality
$\inner x x\land e\le\spp(x)$, again using \ref{lem:stabprop1}:
\[\inner x x\land e=\spp(\inner x x\land e)\le\spp(\inner x x)\land\spp(e)=\spp(x)\land e=\spp(x)\;. \qed\]
\end{proof}

One consequence of this is that the existence of a support (when $Q$ is stably supported) is a property of a pre-Hilbert $Q$-module rather than extra structure. In fact this uniqueness is even ``pointwise'', in the following sense:

\begin{lemma}\label{lem:stabprop2}
Let $Q$ be a stably supported quantale, let $X$ be a (necessarily stably) supported $Q$-module, and let
$x\in X$ and $b\in B$ be such that
\begin{eqnarray*}
b &\le & \langle x,x\rangle\;,\\
x &\le & bx\;.
\end{eqnarray*}
Then we necessarily have $b=\spp (x)$.
\end{lemma}

\begin{proof}
Using the $B$-equivariance of the support of $X$ we obtain
\[\spp (x) \le \spp(bx) =b\land\spp(x)\le b\;,\]
and, conversely,
$b=\spp (b)\le\spp(\langle x,x\rangle)=\spp (x)$. \qed
\end{proof}

Finally, if there exists a Hilbert basis we obtain:

\begin{theorem}
Let $Q$ be a stably supported quantale. Any complete Hilbert $Q$-module is a (necessarily stably) supported $Q$-module.
\end{theorem}

\begin{proof}
Define $\spp(x)=\langle x,x\rangle\land e$ for all $x\in X$, and let $\sections$ be a Hilbert basis of $X$. By definition, in order to verify that $\spp$ is a support it only remains to be seen that $x\le\spp(x)x$ for all $x\in X$. Let then $s\in\sections$. We have, using the properties of the stable support of $Q$,
\begin{eqnarray*}
\langle x,s\rangle&=&\spp(\langle x,s\rangle)\langle x,s\rangle
=(\langle x,s\rangle\langle x,s\rangle^*\land e)\langle x,s\rangle\\
&\le&(\langle x,x\rangle\land e)\langle x,s\rangle
=\spp(x)\langle x,s\rangle\\
&=& \langle\spp(x)x,s\rangle\;.
\end{eqnarray*}
Hence, we have $x\le\spp(x)x$ due to \ref{usefullemma}, and by \ref{lem:necstab} we conclude that $X$ is stably supported. \qed
\end{proof}

\subsection{Groupoid sheaves}\label{sec:backgrpdshvs}

Now we achieve the main aim of this paper, which is to show that if $G$ is an \'etale groupoid then the classifying topos $BG$ is equivalent to $\sets(\opens(G))$ or, equivalently, to $\maps(\opens(G)$-$\hmb)$. We shall do this by showing that the $\opens(G)$-sheaves of section \ref{sec:actionsonsheaves} coincide with the complete Hilbert $\opens(G)$-modules.

\paragraph{From complete Hilbert $Q$-modules to $Q$-sheaves.} We begin by recalling an observation from \cite[Example 4.7-3]{HS2}:

\begin{lemma}\label{HS2lemma}
Let $Q$ be an inverse quantal frame. If $X$ is a complete Hilbert $Q$-module then $\V\sections_X=1$.
\end{lemma}
In order to bypass differences in notation and make this paper more self-contained we include the proof here:
\begin{proof}
First, if $s\in\ipi(Q)$ and $t\in\sections_X$ then $st\in\sections_X$ because for all $x\in X$ we have
\[\langle x,st\rangle st = \langle x,t\rangle s^*st\le\langle x,t\rangle t\le x\;.\]
Hence, for all $t\in\sections_X$ we have $1_Q t=\V_{s\in\ipi(Q)} st\le\V\sections_X$, and thus
\[1_X=\V_{t\in\sections_X}\langle 1_X,t\rangle t\le\V_{t\in\sections_X}1_Q t\le\V\sections_X\;. \qed\]
\end{proof}

\begin{theorem}\label{thm:Qsheaf}
Let $Q$ be an inverse quantal frame. Every complete Hilbert $Q$-module is a $Q$-sheaf.
\end{theorem}

\begin{proof}
Let $X$ be an arbitrary but fixed complete Hilbert $Q$-module, and let us just write $\sections$ for the Hilbert basis $\sections_X$. The action restricts to an action of the base locale $B\subset Q$ and we shall prove that this defines a $B$-sheaf by showing that it has a Hilbert $B$-module structure with respect to which $\sections$ is a Hilbert basis. First, by the \emph{local inner product} will be meant the operation
$\inner--^\ell:X\times X\to B$
defined by
\begin{eq}\label{localinnerproduct}
\inner x y^\ell = \inner x y\land e\;.
\end{eq}%
The operation $\langle -,-\rangle^\ell$ is of course symmetric, and it preserves joins in the left variable because $\langle-,-\rangle$ does and $Q$ is a locale:
\[\left\langle\V_i x_i,y\right\rangle^\ell = \left\langle\V_i x_i,y\right\rangle\land e
=\left(\V_i\inner{x_i}y\right)\land e = \V_i \inner{x_i}y\land e=\V_i{\inner{x_i}y}^\ell\;.\]
We show that $\langle -,-\rangle^\ell$ is $B$-equivariant in the left variable, using \ref{prop:sqf}:
\[\langle bx,y\rangle^\ell = \langle bx,y\rangle\wedge e=
b\langle x,y\rangle\wedge e = b\wedge\langle x,y\rangle = b\wedge e\wedge\langle x,y\rangle = b\land \langle x,y\rangle^\ell\;.\]
Hence, $X$ together with the local inner product is a pre-Hilbert $B$-module. In order to see that this is a $B$-sheaf let us prove that $\sections$ is itself a Hilbert basis for the pre-Hilbert $B$-module structure; that is, we shall prove that for all $x\in X$ we have $x=\V_{s\in\sections} \inner x s^\ell s$. One inequality is trivial:
\[x=\V_{s\in\sections}\inner x s s\ge\V_{s\in\sections}(\inner x s\land e)s=\V_{s\in\sections}\inner x s^\ell s\;.\]
In order to prove the other inequality, first we apply (\ref{spp9}) with $b=\inner x s^\ell$ and $a=\inner s t$:
\begin{eqarray}
\left\langle{\V_{s\in\sections}\inner x s^\ell s},t\right\rangle&=&
\V_{s\in\sections} \inner x s^\ell\inner s t=\V_{s\in\sections}\inner x s^\ell 1\land\inner s t \\
&=&\V_{s\in\sections}(\inner x s\land e)1\land\inner s t\;.\label{tobecont}
\end{eqarray}%
Now recall the inequality (\ref{sqfprop}):
\[
(a\wedge e)1 \ge \V_{yz^*\le a} y\wedge z\;.
\]
Applying this to the right hand side of (\ref{tobecont}) we obtain
\begin{eq}\label{continued}
\V_{s\in\sections}(\inner x s\land e)1\land\inner s t\ge
\V_{s\in\sections}\V_{yz^*\le\inner x s} y\land z\land\inner s t\;.
\end{eq}%
A particular choice of $y$ and $z$ for which $yz^*\le\inner x s$ is to take $y=\inner x t$ and $z=\inner s t$, and thus with these values of $y$ and $z$ the right hand side of (\ref{continued}) is greater than or equal to
\begin{eqnarray*}
\V_{s\in\sections}y\land z\land\inner s t&=&\V_{s\in\sections}\inner x t\land\inner s t\land\inner s t
=\V_{s\in\sections}\inner x t\land\inner s t\\
&=&
\inner x t\land\V_{s\in\sections}\inner s t
=\inner x t\land\left\langle\V_{s\in\sections} s,t\right\rangle\\
&=&\inner x t\land\inner 1 t=\inner x t\;,
\end{eqnarray*}
where the transition to the last line follows from \ref{HS2lemma}.
Hence, we have concluded that $\left\langle\V_{s\in\sections}\inner x s^\ell s,t\right\rangle\ge\inner x t$ for all $t\in\sections$, which finally gives us:
\[\V_{s\in\sections}\inner x s^\ell s\ge x\;. \qed\]
\end{proof}

\paragraph{From $Q$-sheaves to complete Hilbert $Q$-modules.}

We begin with a simple technical lemma:

\begin{lemma}\label{iqaction}
Let $Q$ be an inverse quantal frame and let $X$ be a $Q$-sheaf. The action of $Q$ on $\opens(X)$ restricts to a monoid action of $\ipi(Q)$ on the set of local sections $\sections_X$.
\end{lemma}

\begin{proof}
Let $t\in\sections_X$ and $s\in \ipi(Q)$. We want to prove that $st\in\sections_X$. So let $x\le st$ and let us show that $x=\spp(x)st$. First we note that
$s^*x\le s^*st\le t$, and thus
\begin{eq}\label{littleq1}
\spp(s^*x)t=s^*x
\end{eq}%
because $t$ is a section. Secondly, we have $\spp(x)\le\spp(st)\le\spp(s)=ss^*$, and thus
\begin{eq}\label{littleq2}
ss^*x=x\;.
\end{eq}%
Hence,
since $s=ss^*s$ and both $\spp(x)$ and $ss^*$ belong to the base locale $B\subset Q$, applying (\ref{littleq1})--(\ref{littleq2}) and the equality $\spp(s^* x)=s^*\spp(x)s$ from \ref{lem:openGbundle}-\ref{bspp3} we obtain
\[\spp(x)st=\spp(x)ss^*st=ss^*\spp(x)st=s\spp(s^*x)t=ss^*x=x\;. \qed\]
\end{proof}

\begin{theorem}\label{thm:sheaveseqetales}
Let $Q$ be an inverse quantal frame. Every $Q$-sheaf is a complete Hilbert $Q$-module.
\end{theorem}

\begin{proof}
Let $X$ be an arbitrary but fixed $Q$-sheaf, and let us denote the base locale of $Q$ by $B$.
The proof has two parts: first, we construct a $Q$-set $(I,M)$; then, we prove that the Hilbert $Q$-module $Q^I M$ (\cf\ \ref{lemma:matrixtomod}) is isomorphic to $\opens(X)$.

For the set $I$ we take the set $\sections_X$ of local sections of $X$. And the matrix $M:I\times I\to Q$ is defined by
\begin{eqarray}
\ipi(Q)_{st} &=& \{a\in\ipi(Q)\st \spp(a^*)\le\spp(t)\textrm{ and }at\le s\}\label{mst1}\;,\\
m_{st} &=& \V\ipi(Q)_{st}\label{mst2}\;.
\end{eqarray}%
First we remark that the condition $at\le s$ in (\ref{mst1}) is equivalent to
\begin{eq}\label{mst3}
\spp(at)s=at
\end{eq}%
because, by \ref{iqaction}, $at$ is a local section  [\cf\ (\ref{eq:locsec1})].
Next we mention that the two conditions $\spp(a^*)\le\spp(t)$ and $at\le s$ in (\ref{mst1}) also imply $\spp(a)\le\spp(s)$, since
\begin{eq}\label{doubleproof}
\spp(a)=\spp(a\spp(a^*))\le\spp(a\spp(t))=\spp(at)\le\spp(s)\;.
\end{eq}%
In addition, for all $a\in\ipi(Q)_{st}$ we have $a^*s\le t$ because from (\ref{doubleproof}) we obtain $\spp(a)\le\spp(at)$, and
from $at\le s$ and (\ref{mst3}) we obtain $aa^*s=\spp(a)s\le\spp(at)s=at$, whence
\[a^*s=a^* a a^* s\le a^* at\le t\;.\]
It follows that $a^*\in\ipi(Q)_{ts}$, and we conclude that $M^*=M$. Furthermore, for all $a\in\ipi(Q)_{st}$ and $b\in\ipi(Q)_{tu}$ we have $ab\in \ipi(Q)_{su}$:
\begin{eqnarray*}
\spp((ab)^*)&=&\spp(b^*a^*)\le\spp(b^*)\le u\\
abu&\le&at\le s\;.
\end{eqnarray*}
Hence, $m_{st} m_{tu}\le m_{su}$. This shows that we have $M^2\le M$, which is equivalent to $M^2=M$ because $Q$ is a supported quantale (\cf\ \ref{thmQsets} in the appendix), and thus $(I,M)$ is a $Q$-set.

Therefore, by \ref{lemma:matrixtomod}, we have a Hilbert $Q$-module $Q^I M$, whose inner product is the dot product of $Q^I$, with a Hilbert basis $\sections$ consisting of the rows of $M$. As usual (\cf\ section \ref{sec:mvhb}), for each $s\in I$ we denote the $s$-row of $M$ by $\tilde s$.  The local inner product $\inner --^\ell$ of $Q^I M$, as defined by (\ref{localinnerproduct}), satisfies, for all $s,t\in I$,
\begin{eqnarray*}
\inner {\tilde s}{\tilde t}^\ell &=& \tilde s\cdot \tilde t \land e = m_{st}\land e\\
&=&\V\{a\land e\st a\in \ipi(Q)\textrm{ and }\spp(a^*)\le\spp(t)\textrm{ and }at\le s\}\\
&\le&\V\{a\land e\st a\in \ipi(Q)\textrm{ and }\spp((a\land e)^*)\le\spp(t)\textrm{ and }(a\land e)t\le s\}\\
&=&\V\{b\in B\st\spp(b^*)\le\spp(t)\textrm{ and }bt\le s\}\\
&=&\V\{b\in B\st\spp(b^*)\le\spp(t)\textrm{ and }bt\le s\}\land e\\
&\le&\V\{a\in \ipi(Q)\st\spp(a^*)\le\spp(t)\textrm{ and }at\le s\}\land e\\
&=&m_{st}\land e = \inner{\tilde s}{\tilde t}^\ell\;.
\end{eqnarray*}
Therefore all the expressions in the above derivation are equal, and we obtain
\begin{eqnarray*}
\inner{\tilde s}{\tilde t}^\ell &=& \V\{b\in B\st\spp(b^*)\le\spp(t)\textrm{ and }bt\le s\}\\
&=&\V\{b\in B\st b\le\spp(t)\textrm{ and }bt\le s\}\\
&=&\spp(s\land t)\;.
\end{eqnarray*}
Since $\spp(s\land t)$ is the $B$-valued inner product of $\opens(X)$ regarded as a $B$-sheaf,
we conclude, by \ref{equivrelQQhmb} (or \cite[Lemma 5]{RR}),  that $\opens(X)\cong B^I M^\ell$, where
$M^\ell$ is the matrix defined by
$(M^\ell)_{st}=\spp(s\land t)$. Since this matrix is also that of the local inner product of $Q^I M$, it follows that $B^I M^\ell\cong Q^I M$, and thus $\opens(X)$ is a complete Hilbert $Q$-module.
\qed
\end{proof}

\begin{remark}
We can provide a more geometric interpretation of the formula (\ref{mst2}) in terms of an \'etale groupoid $G$ rather than its quantale $\opens(G)$, whereby $m_{st}$ is the union of all the local bisections $a$ that satisfy the following three conditions (\cf\ \ref{exmtopologicalgroupoid}):
\begin{itemize}
\item the domain of $a$ is contained in the domain of the local section $s$;
\item the image of $r\circ a$ is contained in the domain of the local section $t$;
\item $a$ acts on $t$ yielding a subsection of $s$.
\end{itemize}
Hence, the logical interpretation of $G$-sheaves via $\opens(G)$-sets is a generalization of that of frame-valued sets, with the truth values now defining ``equality'' of local sections up to local translations rather than just restriction.
\end{remark}

\paragraph{The classifying topos of an \'etale groupoid.}

From \cite[Theorem 4.1]{HS2}, which applies to modular quantal frames, and hence also to inverse quantal frames (\cf\ \ref{thm:iqfaremod}), it follows that if $Q$ is an inverse quantal frame then a left $Q$-module is a complete Hilbert $Q$-module with respect to at most one inner product. Hence, being such a module is a property rather than extra structure. Similarly, being a $Q$-sheaf is a property, and the results that we have just obtained show that the two properties coincide.

Hence, in order to conclude our comparison of the classifying topos $BG$ of an \'etale groupoid $G$ and $\sets(\opens(G))$ we only need to see how the morphisms of $\opens(G)$-$\sh$ and $\maps(\opens(G)$-$\hmb)$ relate to each other:

\begin{lemma}\label{lem:main}
Let $Q$ be an inverse quantal frame. The sheaf homomorphisms of $Q$-sheaves (\cf\ \ref{def:shhom}) coincide with the arrows of $\maps(Q$-$\hmb)$.
\end{lemma}

\begin{proof}
Let $X$ and $Y$ be $Q$-sheaves and let $\varphi:\opens(X)\to \opens(Y)$ be a sheaf homomorphism of $Q$-sheaves. This is also a sheaf homomorphism of $B$-sheaves, where $B\subset Q$ is the base locale of $Q$, and thus we have $\varphi=f_!$ for a map of $B$-sheaves $f:X\to Y$ (\cf\ \ref{exm:BsheavesasHilbertmodules} and section \ref{sec:localesheavesasmodules}). The inner products of $X$ and $Y$ as $B$-sheaves are defined by the formula (\ref{localinnerproduct}), which we recall: \[\inner z{z'}^\ell=\inner z {z'}\land e\;.\]
Hence, the adjoint of $\varphi$, which is given by $\inner{\varphi(x)}y=\inner x{\varphi^\dagger(y)}$, also satisfies
${\inner{\varphi(x)}y}^\ell={\inner x{\varphi^\dagger(y)}}^\ell$, showing that $\varphi^{\dagger}$ is the adjoint of $\varphi$ also with respect to the $B$-sheaf structures of $X$ and $Y$. Therefore we have $f^*=\varphi^\dagger$, showing that $\varphi^\dagger$ is right adjoint to $\varphi$; that is, $\varphi$ is a morphism in
$\maps(Q$-$\hmb)$.

Conversely, let $\varphi:\opens(X)\to\opens(Y)$ be a morphism in $\maps(Q$-$\hmb)$.
By the same reasoning as above, this is also a morphism in $\maps(B$-$\hmb)$, and thus $\varphi$ is a sheaf homomorphism of $B$-sheaves. Since it is also a homomorphism of $Q$-modules, it is a sheaf homomorphism of $Q$-sheaves. \qed
\end{proof}

Our main result follows immediately:

\begin{theorem}\label{thm:main}
For any \'etale groupoid $G$, the category $\sets(\opens(G))$ (equivalently, $\maps(\opens(G)$-$\hmb)$) is a Grothendieck topos and it is equivalent to the classifying topos $BG$ (hence it is an \'etendue).
\end{theorem}

\begin{corollary}
Any \'etendue is equivalent to $\sets(Q)$ (equivalently, $\maps(Q$-$\hmb)$) for some inverse quantal frame $Q$.
\end{corollary}

\section{Concluding remarks}\label{sec:end}

Our results add credibility to the notion of sheaf for involutive quantales that is prevalent in the literature, due to the following two reasons:
\begin{itemize}
\item for an \'etale groupoid $G$ the $\opens(G)$-sets give us the classifying topos of $G$;
\item for involutive quantales with base locales another natural notion of ``equivariant sheaf'' is obtained by generalizing the definition of $Q$-sheaf of \ref{def:Qsheaf} and, as we have seen, for an inverse quantal frame $Q$ the resulting category $Q$-$\sh$ is equivalent to $\sets(Q)$.
\end{itemize}

Another natural criterion for assessing the value of any notion of quantale sheaf is that of whether the categories of sheaves are toposes. This topic has been very recently addressed in \cite[Proposition 4.4.12]{HeymansPhD} by showing that the category of $Q$-sheaves is a Grothendieck topos if and only if $Q$ is a so-called \emph{Grothendieck quantale}, by which is meant a modular quantal frame that satisfies an additional simple algebraic condition. (In particular, inverse quantal frames are Grothendieck quantales.)

This criterion leaves out stable quantal frames which, as we have seen (\cf\ \ref{exm:ssqnotmod}), are more general than modular quantal frames but nevertheless form an algebraically nice class of quantales. While it is not inconceivable that a slightly more restricted notion of sheaf might be appropriate for stable quantal frames, especially if this yields toposes of sheaves, presently not much motivation seems to exist in order to pursue this question due to the unavailability of natural examples of stable quantal frames beyond those that are inverse quantal frames (despite which there may be reasons for addressing even more general stably supported quantales, as hinted at in the beginning of section \ref{sec:ssq}). Nevertheless we remark that any stable quantal frame $Q$ has a base locale $B$, and that in those particular cases where $Q$ is an inverse quantal frame every complete Hilbert $Q$-module $X$ is also a $B$-sheaf, as has been proved in Theorem \ref{thm:Qsheaf}. This is a rather natural property, but it depends crucially on the fact that the Hilbert sections of $X$ cover $X$, which is false in general for Hilbert modules on arbitrary stable quantal frames, as example \ref{exm:hilbsec3} shows. Hence, if $Q$ is a stable quantal frame, we regard this cover condition as an example of an axiom that may make sense adding to the definition of sheaf for $Q$.

\section{Appendix}\label{sec:shmat}

We provide a brief survey of the variants of the notion of quantale-valued set that exist in the literature in order to make clear that they are all equivalent in the cases that interest us in this paper, namely inverse quantal frames, and even more generally in the case of stably Gelfand quantales, which encompass the known examples of quantales-as-spaces. In this way we intend to convey an idea of robustness of the definitions that surround the notion of sheaf for involutive quantales and that such sheaves can indeed be taken to be quantale-valued sets, or, equivalently, complete Hilbert modules. We shall also describe completeness of quantale-valued sets and compare it with the natural notion of completeness that arises from complete Hilbert modules, concluding that the two notions coincide.

\paragraph{Stably Gelfand quantales.}

Mulvey \cite{Curacao,MP0} has noticed that the notion of \emph{Gelfand quantale} (referred to as ``localic quantale'' in \cite{Curacao}) plays a relevant role in the study of quantales of C*-algebras. A Gelfand quantale is an involutive quantale whose right sided elements (that is, those elements $a$ such that $a1\le a$) satisfy the regularity (or ``Gelfandness'') condition $aa^*a=a$. Later, he and Ramos \cite{MulJoel,JoelPhD} have put forward the stronger notion of \emph{locally Gelfand quantale}, which is important in connection with their study of quantal sets and sheaves on involutive quantales; in a locally Gelfand quantale, all the elements $a$ such that $a\le p$ and $ap\le a$ for some projection $p$ are required to be regular. An even stronger notion is the following:

\begin{definition}\label{stablyGelfQ}
By a \emph{stably Gelfand quantale} is meant an involutive quantale $Q$ such that, for all $a\in Q$, if $aa^*a\le a$ then $aa^*a = a$ (in other words, if $a$ is ``stable'' under the operation $a\mapsto aa^*a$ then it is a ``Gelfand element'', hence the terminology).
\end{definition}

The main examples of ``quantales-as-spaces'' are stably Gelfand:

\begin{example}\label{exm:quantalesasspaces}
\begin{enumerate}
\item Any involutive quantale $Q$ that satisfies the condition $a\le aa^*a$ for all $a\in Q$ is stably Gelfand. This includes supported quantales and, hence, modular quantales. It also includes the involutive quantales associated to localic or topological open groupoids.
\item\label{maxaissg} The involutive quantale $\Max A$ of a C*-algebra $A$ is stably Gelfand because it satisfies the condition $VV^*VV^*\le VV^*\Rightarrow VV^*V=V$ for all $V\in\Max A$. To some extent this example enables us to generalize to arbitrary sub-C*-algebras the constructions, due to Kumjian and Renault \cite{Kumjian,Renault}, of \'etale groupoids from suitable commutative sub-C*-algebras of C*-algebras (``diagonals''), via a construction that associates a localic \'etale groupoid to each projection of a stably Gelfand quantale. See \cite{ICslides} for this construction and a proof that $\Max A$ is stably Gelfand.
\end{enumerate}
\end{example}

The first appearance of stably Gelfand quantales in writing seems to be in \cite{Garraway}, where these quantales (in fact quantaloids) are called \emph{pseudo-rightsided} and play an important role in unifying variants of the notion of quantale-valued set, as we shall see.

\paragraph{Quantale-valued sets.}

Let $Q$ be an involutive quantale. As we have mentioned in section \ref{sec:mvhb}, there is a logical interpretation of $Q$-sets (\cf\ comments after \ref{Qsets}). This point of view is emphasized by Mulvey and Ramos \cite{MulJoel,JoelPhD}, who adopt the style of notation introduced earlier for frame-valued sets
\citelist{ \cite{Higgs} \cite{FS} \cite{Borceux}*{Section 2} \cite{elephant}*{pp.\ 502--513}} (subsequently applied to right-sided quantales) and define a \emph{quantal set} over $Q$ to be a set $I$ equipped with mappings
\[E:I\to Q\hspace*{1cm} \qseq \cdot\cdot:I\times I\to Q\;,\]
referred to respectively as \emph{extent} and \emph{equality}, satisfying the following conditions for all $\alpha,\beta,\gamma\in I$:
\begin{enumerate}
\item $E(\alpha) = \qseq \alpha\alpha$
\item $\qseq\alpha\beta^*=\qseq\beta\alpha$
\item $\qseq\alpha\beta\qseq\beta\gamma\le\qseq\alpha\gamma$
\item $\qseq\alpha\beta \le E(\alpha)\qseq\alpha\beta$
\item $\qseq\alpha\beta \le \qseq\alpha\beta E(\beta)$.
\end{enumerate}
We remark that $E$ is used for convenience only, since it is derived from equality.
The third condition (transitivity of equality) ensures that the matrix $A$ defined by $a_{\alpha\beta}=\qseq\alpha\beta$ satisfies $AA\le A$, and the converse inequality holds due to the first and fourth (or fifth) conditions. Hence, since the mapping $E$ is redundant, a quantal set is the same as a strict $Q$-set in sense of Garraway \cite{Garraway} and Gylys \cite{G01}:

\begin{definition}\label{strictQsets}
Let $Q$ be an involutive quantale.
A $Q$-set is \emph{strict} if for all $\alpha,\beta\in I$ we have $a_{\alpha\alpha}a_{\alpha\beta}=a_{\alpha\beta}$ (and thus also $a_{\beta\alpha}a_{\alpha\alpha}=a_{\beta\alpha}$).
\end{definition}

Mulvey and Ramos further define a \emph{Gelfand quantal set} to be a quantal set satisfying the following condition for all $\alpha$ and $\beta$:
\[
\qseq\alpha\beta\le\qseq\alpha\beta\qseq\alpha\beta^*\qseq\alpha\beta\;.
\]

\begin{theorem}\label{thmQsets}
If $Q$ is stably Gelfand all the above variants of $Q$-valued set, namely $Q$-sets, quantal sets (= strict $Q$-sets), and Gelfand quantal sets, coincide. In particular, if $Q$ is a frame we obtain the usual notion of frame-valued set.
\end{theorem}

\begin{proof}
Every $Q$-set $(I,A)$ satisfies $AA^*A\le A$ and thus $a_{\alpha\beta}a_{\alpha\beta}^*a_{\alpha\beta}\le a_{\alpha\beta}$ for all $\alpha,\beta\in I$. Hence, if $Q$ is stably Gelfand we have $a_{\alpha\beta}a_{\alpha\beta}^*a_{\alpha\beta}= a_{\alpha\beta}$ for all $\alpha,\beta\in I$ and thus the $Q$-set is strict \cite[Lemma 4.1]{Garraway}:
\[a_{\alpha\beta}=a_{\alpha\beta}a_{\alpha\beta}^*a_{\alpha\beta}=a_{\alpha\beta}a_{\beta \alpha}a_{\alpha\beta}\le a_{\alpha \alpha}a_{\alpha\beta}\le a_{\alpha\beta}\;.\]
If $\Omega$ is a frame, in matrix language its frame-valued sets are the $\Omega$-valued matrices $A$ that satisfy $AA\le A=A^T$, and thus they coincide with $\Omega$-sets according to \ref{Qsets} because frames, seen as quantales with trivial involution, are stably Gelfand. \qed
\end{proof}

\paragraph{Maps.} Similarly to quantale-valued sets, there are strict notions of map:

\begin{definition}
Let $Q$ be an involutive quantale, and let $F:(I,A)\to (J,B)$ be a map of $Q$-sets.
The map $F$ is said to be \emph{strict} if it satisfies the conditions
\begin{eqarray}
f_{\beta\alpha}&=& f_{\beta\alpha}a_{\alpha\alpha}\label{dstrict}\\
f_{\beta\alpha}&=& b_{\beta\beta}f_{\beta\alpha}\label{cstrict}
\end{eqarray}%
for all $\alpha\in I$ and $\beta\in J$.
\end{definition}

\begin{lemma}\label{lemmastrict}
Let $F:(I,A)\to (J,B)$ be a map of $Q$-sets. If $(I,A)$ is strict then $F$ satisfies (\ref{dstrict}). If $(J,B)$ is strict then $F$ satisfies (\ref{cstrict}).
\end{lemma}

\begin{proof}
Assume that $(I,A)$ is strict. Then for all $\alpha\in I$ and $\beta\in J$ we have
\[f_{\beta\alpha}=(FA)_{\beta\alpha}=\V_{\alpha'} f_{\beta\alpha'}a_{\alpha'\alpha}=\V_{\alpha'} f_{\beta\alpha'}a_{\alpha'\alpha}a_{\alpha\alpha}=(FA)_{\beta\alpha}a_{\alpha\alpha}=f_{\beta\alpha}a_{\alpha\alpha}\;.\]
The strictness condition (\ref{cstrict}) follows from the strictness of $(J,B)$ is a similar way using the equality $F=BF$. \qed
\end{proof}

Mulvey and Ramos \cite{MulJoel,JoelPhD} have proposed a notion of map which is formally equivalent to the above notion of strict map, and thus also equivalent to a general map because their quantal sets are strict $Q$-sets. They have also defined a map $F$ to be \emph{Gelfand} if the condition
\[f_{\beta\alpha}\le f_{\beta\alpha}f_{\beta\alpha}^*f_{\beta\alpha}\]
holds for all $\alpha$ and $\beta$, and they have proved that, if $Q$ is locally Gelfand, the composition of two Gelfand maps is a Gelfand map. In this way a category of Gelfand quantal sets and Gelfand maps is defined.

But, again, if $Q$ is stably Gelfand (rather than just locally Gelfand) there are additional simplifications, since any map is necessarily strict and Gelfand:

\begin{theorem}\label{notionsofmap}
If $Q$ is stably Gelfand all the above variants of map of $Q$-set coincide. In particular, if $Q$ is a frame we obtain the usual notion of map of frame-valued sets.
\end{theorem}

\begin{proof}
Any map $F:(I,A)\to(J,B)$ satisfies $FF^*F\le BF=F$, and thus
$f_{\beta\alpha}f_{\beta\alpha}^*f_{\beta\alpha}\le f_{\beta\alpha}$ for all $\alpha$ and $\beta$. Hence, if $Q$ is stably Gelfand we obtain
$f_{\beta\alpha}f_{\beta\alpha}^*f_{\beta\alpha}= f_{\beta\alpha}$, showing that $F$ is a Gelfand map. And $F$ is strict due to \ref{thmQsets} and \ref{lemmastrict}.
Finally, in matrix language, for a frame $\Omega$ a map of $\Omega $-valued sets $F:(I,A)\to (J,B)$ is \cite{elephant}*{pp.\ 502--513} an $\Omega $-valued matrix $F: J\times I\to \Omega$ satisfying
\begin{eqnarray*}
f_{\beta\alpha}&\le&a_{\alpha\alpha}\wedge b_{\beta\beta}\\
f_{\beta\alpha}\wedge a_{\alpha\alpha'}\wedge b_{\beta\beta'}&\le&f_{\beta'\alpha'}\\
f_{\beta\alpha}\wedge f_{\beta'\alpha}&\le& b_{\beta\beta'}\\
a_{\alpha\alpha}&\le&\V_{\beta} f_{\beta\alpha}
\end{eqnarray*}
for all $\alpha\in I$ and $\beta\in J$, and it is easy to see that this is just the definition of strict map in frame language: strictness is the first condition; together with the second one it gives us $BFA= F$; the third one is $FF^*\le B$; and the fourth one is $a_{\alpha\alpha}\le (F^*F)_{\alpha\alpha}$, which together with the strictness of $(I,A)$ gives us $A\le F^*F$. \qed
\end{proof}

\paragraph{Complete quantale-valued sets.} In order to describe completeness of quantale-valued sets it will be convenient to depart from our usual convention and identify mappings $S:I\to Q$ with column matrices instead of row matrices --- with the adjoint $S^*$ being regarded as a row matrix. The following definition is adapted from \cite{Garraway,G01}.

\begin{definition}\label{def:singleton}
Let $Q$ be an involutive quantale, and let $(I,A)$ be a $Q$-set. By a \emph{singleton map} (or simply a \emph{singleton}) of $(I,A)$ is meant a mapping $S:I\to Q$ for which there exists a projection $q\in Q$ such that $S$, regarded as a column matrix $I\times\{*\}\to Q$, defines a map of $Q$-sets $S:[q]\to (I,A)$ where $[q]$ is the $Q$-set defined by the $\{*\}\times \{*\}$ matrix with single entry $q$.
\end{definition}

In other words, $S:I\to Q$ is a singleton if and only if it satisfies, for some projection $q$, the following conditions for all $\alpha,\beta\in I$:
\begin{eqarray}
s_\alpha &=& s_\alpha q \label{sing1}\\
q&\le& S^*S=\V_{\gamma\in I} s^*_{\gamma}s_{\gamma} \label{sing2}\\
s_\alpha&=&\V_{\gamma\in I} a_{\alpha\gamma}s_\gamma \label{sing3}\\
s_\alpha s^*_\beta&\le& a_{\alpha\beta} \label{sing4}\;.
\end{eqarray}%
Again there is a logical interpretation, namely $S$ can be regarded as a ``subset'' of $I$ with $s_\alpha$ being the truth value of the assertion that $\alpha$ belongs to $S$. In particular, (\ref{sing3}) implies
\begin{eq}\label{sing3prime}
a_{\alpha\beta}s_\beta\le s_\alpha\;,
\end{eq}%
which can be read ``if $\alpha$ equals $\beta$ and $\beta$ is in $S$ then so is $\alpha$'', and (\ref{sing4}) can be read ``if $\alpha$ is in $S$ and $\beta$ is in $S$ then $\alpha$ equals $\beta$''; the latter expresses the idea that $S$ is a ``singleton subset''.

We remark that if $(I,A)$ is a strict $Q$-set then, multiplying both sides of (\ref{sing3}) on the left by $a_{\alpha\alpha}$, we obtain a strictness condition for singletons:
\begin{eq}\label{eq:strictsingleton}
a_{\alpha\alpha}s_{\alpha}=s_{\alpha}\;.
\end{eq}%
We also remark that the conjunction of (\ref{sing3prime}) and (\ref{eq:strictsingleton}) is equivalent to (\ref{sing3}).

As before, simplifications arise if $Q$ is stably Gelfand:

\begin{theorem}\label{RMsingleton}
Let $Q$ be a stably Gelfand quantale, and let $S:I\to Q$ be a mapping. Then $S$ is a singleton if and only if it satisfies (\ref{sing4}) and (\ref{sing3prime}). Furthermore, if $S$ is a singleton we have
\begin{eq}\label{eq:gelfandsingleton}
s_{\alpha}s_{\alpha}^*s_{\alpha}=s_{\alpha}
\end{eq}%
for all $\alpha\in I$.
\end{theorem}

\begin{proof}
If $S$ is a singleton it satisfies (\ref{sing4}) by definition, and (\ref{sing3prime}) follows from (\ref{sing3}), as we have already mentioned above. Conversely, assume that $S$ satisfies (\ref{sing4}) and (\ref{sing3prime}), and let $q=S^*S$. This projection trivially satisfies (\ref{sing2}). It also satisfies (\ref{sing1}) because, by \ref{notionsofmap}, $S$ is necessarily a Gelfand map:
\begin{eqarray}s_\alpha q &=& s_\alpha\V_{\gamma} s^*_\gamma s_\gamma\ge s_\alpha s^*_\alpha s_\alpha\ge s_\alpha\;; \label{firstline}\\
s_\alpha q &=&\V_{\gamma} s_\alpha s^*_\gamma s_\gamma
\le\V_\gamma a_{\alpha\gamma} s_{\gamma}\le s_\alpha\;.\label{secondline}
\end{eqarray}%
In addition, we conclude that all the above inequalities are in fact equalities, and thus from (\ref{secondline}) we obtain (\ref{sing3}). Hence, $S$ is a singleton. Similarly, the conclusion that singletons satisfy (\ref{eq:gelfandsingleton}) follows from (\ref{firstline}). \qed
\end{proof}

Mulvey and Ramos \cite{MulJoel,JoelPhD} define singletons in a different way. According to their definition a singleton is, in matrix language, a mapping $S:I\to Q$ that satisfies (\ref{sing4}) and (\ref{sing3prime}) together with the condition $s_{\alpha}\le s_{\alpha}s_{\alpha}^*s_{\alpha}$. Hence, if $Q$ is a stably Gelfand quantale their notion of singleton is equivalent to the one we have been using. In particular, if $Q$ is a locale this coincides with the standard notion of singleton for frame-valued sets \citelist{ \cite{Higgs} \cite{FS} \cite{Borceux}*{Section 2} \cite{elephant}*{pp.\ 502--513}}.

Let $(I,A)$ be a $Q$-set, again with $Q$ stably Gelfand.
It is immediate that every column of $A$ is a singleton. A $Q$-set is said to be \emph{complete} if each of its singletons arises in this way from a unique column of $A$, and there is a notion of \emph{completion} of $(I,A)$, which is the complete $Q$-set that consists of the
set of all the singletons of $(I,A)$ equipped with the matrix $\widehat A$ defined by a dot product:
\[\hat a_{ST}=S^*T=\V_{\alpha\in I}s_\alpha^*t_\alpha\;.\]
If $B$ is a frame, the difference between arbitrary $B$-sets and the complete ones is that the latter are more canonical in the sense that there is a functor from $\sh(B)$ to $\sets(B)$ that to each sheaf assigns a $B$-set that is complete. However, it is well known that the category $\sets(B)$ is equivalent to its full subcategory of complete $B$-sets (and equivalent to $\sh(B)$). Analogously, for a stably Gelfand quantale $Q$ the category $\sets(Q)$ is equivalent to its full subcategory of complete $Q$-sets. This fact is proved by Garraway \cite[Section 4]{Garraway} and it shows that, to a large extent, for stably Gelfand quantales the completeness of $Q$-sets is irrelevant.

\paragraph{Completeness via Hilbert modules.}

Now let us compare, for a stably Gelfand quantale $Q$, the notion of complete Hilbert $Q$-module with the notion of complete $Q$-set, namely seeing that the former is more canonical than $Q$-sets exactly for the same reason that complete $Q$-sets are: the completion of a $Q$-set $(I,A)$ coincides (up to renaming of matrix indices) with the $Q$-set associated to the Hilbert $Q$-module $Q^I A$.

\begin{theorem}
Let $Q$ be a stably Gelfand quantale, let $(I,A)$ be a $Q$-set, and let $S:I\to Q$ be an arbitrary mapping. The following statements are equivalent:
\begin{enumerate}
\item $S$ is a singleton of $(I,A)$;
\item $S^*$ is a Hilbert section of the Hilbert $Q$-module $Q^I A$.
\end{enumerate}
Moreover, letting $(\widehat I,\widehat A)$ be the completion of $(I,A)$ we have, for all $S,T\in \widehat I$, the following equality:
\begin{equation}\label{Avsip}
(\widehat A)_{ST} = \langle S^*,T^*\rangle\;.
\end{equation}
\end{theorem}

\begin{proof}
First assume that $S$ is a singleton; that is, $S$ satisfies the following two conditions:
\begin{eqnarray*}
AS &=& S\;,\\
SS^* &\le & A\;.
\end{eqnarray*}
Hence, $S^*$ is in $Q^I A$, for $S^*A = S^* A^* = (AS)^* = S$.
Moreover, for any other mapping $\vect v\in Q^I A$ (regarded as a row matrix) we have
\[\langle \vect v, S^*\rangle S^* = \vect v S S^* \le \vect v A=\vect v\;,\]
and thus $S^*$ is a Hilbert section.

Now assume that $S^*$ is a Hilbert section of $Q^I A$ (regarded as a row matrix). The condition $AS=S$ follows from $S^* A=S^*$, and the Hilbert section condition gives us $\tilde s S S^*\le \tilde s$ for all $s\in I$:
\[\tilde s S S^* = \langle\tilde s, S^*\rangle S^*\le\tilde s\;.\]
Hence, we obtain $ASS^*\le A$, \ie, $SS^*\le A$, and thus $S$ is a singleton.

Finally, (\ref{Avsip}) is immediate:
$(\widehat A)_{ST}=S^* T = S^*T^{**} = \langle S^*,T^*\rangle$. \qed
\end{proof}

~\\
{\sc
Centro de An\'alise Matem\'atica, Geometria e Sistemas Din\^amicos
Departamento de Matem\'{a}tica, Instituto Superior T\'{e}cnico\\
Universidade T\'{e}cnica de Lisboa\\
Av.\ Rovisco Pais 1, 1049-001 Lisboa, Portugal}\\
{\it E-mail:} {\sf pmr@math.ist.utl.pt}
\end{document}